\documentclass[letterpaper, 11pt,  reqno]{amsart}

\usepackage{amsmath,amssymb,amscd,amsthm,amsxtra, esint}
\usepackage{tikz} 

\usepackage{mathtools}
\usepackage{color}
\usepackage[implicit=true]{hyperref}

\usepackage{cases}

\usepackage[left=32mm, right=32mm, 
bottom=27mm]{geometry}

\allowdisplaybreaks[2]

\usepackage{color}

\definecolor{gr}{rgb}   {0.,   0.69,   0.23 }
\definecolor{bl}{rgb}   {0.,   0.5,   1. }
\definecolor{mg}{rgb}   {0.85,  0.,    0.85}
\definecolor{yl}{rgb}   {0.8,  0.7,   0.}
\definecolor{or}{rgb}  {0.7,0.2,0.2}

\usetikzlibrary{shapes.misc}
\usetikzlibrary{shapes.symbols}
\usetikzlibrary{shapes.geometric}

\tikzset{
	ddot/.style={circle,fill=white,draw=black,inner sep=0pt,minimum size=0.8mm},
	>=stealth,
	}

\tikzset{
	ddot2/.style={circle,fill=black,draw=black,inner sep=0pt,minimum size=0.8mm},
	>=stealth,
	}

\newtheorem{theorem}{Theorem} [section]

\newtheorem{lemma}[theorem]{Lemma}
\newtheorem{proposition}[theorem]{Proposition}
\newtheorem{remark}[theorem]{Remark}

\DeclareMathOperator*{\supp}{supp}

\newcommand{\noi}{\noindent}
\newcommand{\Z}{\mathbb{Z}}
\newcommand{\R}{\mathbb{R}}

\newcommand{\T}{\mathbb{T}}

\let\P= \undefined
\newcommand{\P}{\mathbf{P}}

\newcommand{\E}{\mathbb{E}}

\renewcommand{\L}{\mathcal{L}}

\newcommand{\F}{\mathcal{F}}

\def\norm#1{\|#1\|}

\newcommand{\al}{\alpha}
\newcommand{\be}{\beta}
\newcommand{\dl}{\delta}

\newcommand{\nb}{\nabla}

\newcommand{\Dl}{\Delta}
\newcommand{\eps}{\varepsilon}
\newcommand{\kk}{\kappa}
\newcommand{\g}{\gamma}
\newcommand{\G}{\Gamma}
\newcommand{\ld}{\lambda}

\newcommand{\s}{\sigma}

\newcommand{\ft}{\widehat}

\newcommand{\cj}{\overline}

\newcommand{\dt}{\partial_t}

\newcommand{\ta}{\theta}

\renewcommand{\l}{\ell}
\renewcommand{\o}{\omega}
\renewcommand{\O}{\Omega}

\newcommand{\les}{\lesssim}
\newcommand{\ges}{\gtrsim}

\newcommand{\jb}[1]
{\langle #1 \rangle}

\newcommand{\ind}{\mathbf 1}

\newcommand{\D}{\mathcal{D}}

\newcommand{\N}{\mathbb{N}}

\renewcommand{\H}{\mathcal{H}}

\newtheorem*{ackno}{Acknowledgements}

\newcommand{\rhoo}{\vec{\rho}}
\newcommand{\muu}{\vec{\mu}}
\newcommand{\deff}{\stackrel{\textup{def}}{=}}

\newcommand{\Id}{\textup{Id}}

\numberwithin{equation}{section}
\numberwithin{theorem}{section}

\begin{document}
\baselineskip = 15pt

\title[Global dynamics for   2-$d$ stochastic NLW]
{Global dynamics for the two-dimensional stochastic nonlinear wave equations}

\author[M.~Gubinelli, H.~Koch, T.~Oh, and L.~Tolomeo]
{Massimiliano Gubinelli, Herbert Koch,  Tadahiro Oh,
and Leonardo Tolomeo}

\address{
Massimiliano Gubinelli\\
Hausdorff Center for Mathematics \&  Institut f\"ur Angewandte Mathematik\\
 Universit\"at Bonn\\
Endenicher Allee 60\\
D-53115 Bonn\\
Germany}
\email{gubinelli@iam.uni-bonn.de}

\address{Herbert Koch\\
Mathematisches Institut\\
 Universit\"at Bonn\\
Endenicher Allee 60\\
D-53115 Bonn\\
Germany
}

\email{koch@math.uni-bonn.de}

\address{
Tadahiro Oh, School of Mathematics\\
The University of Edinburgh\\
and The Maxwell Institute for the Mathematical Sciences\\
James Clerk Maxwell Building\\
The King's Buildings\\
Peter Guthrie Tait Road\\
Edinburgh\\ 
EH9 3FD\\
 United Kingdom}

\email{hiro.oh@ed.ac.uk}

\address{
Leonardo Tolomeo\\ 
The University of Edinburgh\\
and The Maxwell Institute for the Mathematical Sciences\\
James Clerk Maxwell Building\\
The King's Buildings\\
Peter Guthrie Tait Road\\
Edinburgh\\ 
EH9 3FD\\
 United Kingdom
 and 
Mathematical Institute, Hausdorff Center for Mathematics, Universit\"at Bonn, Bonn, Germany}

\email{tolomeo@math.uni-bonn.de}



\begin{abstract}
We study global-in-time dynamics 
of the stochastic nonlinear wave equations (SNLW) with 
an additive space-time white noise forcing, 
posed on the two-dimensional torus.
Our goal in this paper is two-fold.~(i) By introducing a   hybrid argument, 
combining  the  $I$-method
 in the stochastic setting with a Gronwall-type argument,
we first prove global well-posedness
of the  (renormalized) cubic SNLW in the defocusing case.
Our argument yields a double exponential growth bound on 
the Sobolev norm of a solution.
(ii)~We then study  the stochastic damped nonlinear wave equations (SdNLW) in the defocusing case.
In particular, by applying Bourgain's invariant measure argument, 
we prove  almost sure global well-posedness of the 
(renormalized) defocusing
SdNLW with respect to the Gibbs measure 
and invariance of the Gibbs measure.

\end{abstract}

%
\maketitle
\tableofcontents

\baselineskip = 14pt

\newpage

\section{Introduction}

\subsection{Stochastic nonlinear wave equations}

In \cite{GKO1}, the first three authors  studied 
the following stochastic nonlinear wave equations (SNLW)  on 
the two-dimensional torus $\T^2 = (\R/ \Z)^2$
with an additive space-time white noise forcing:
\begin{align}
\begin{cases}
\dt^2 u + (1 -  \Dl)  u   + u^k = \xi\\
(u, \dt u) |_{t = 0} = (\phi_0, \phi_1)
\end{cases}
\quad (x, t) \in \T^2\times \R_+,
\label{SNLW1}
\end{align}

\noi
where $k\geq 2$ is an integer and 
$\xi(x, t)$ denotes a (Gaussian) space-time white noise on $\T^2\times \R_+$.
In the following, we  restrict our attention to the real-valued setting.
By introducing an appropriate time-dependent renormalization, 
they proved local well-posedness
of (the renormalized version of) SNLW \eqref{SNLW1}
with (almost) critical initial data.
Our main goal in this paper is
to construct global-in-time dynamics
to  SNLW  in the following two settings:

\medskip

\begin{itemize}
\item[(i)]
When $k = 3$, we introduce a  hybrid argument, 
combining  the so-called $I$-method \cite{CKSTT1, CKSTT2}
and a Gronwall-type argument \cite{BT2}, 
and prove global well-posedness of~\eqref{SNLW1}.
See Subsection~\ref{SUBSEC:GWP0}.

\smallskip

\item[(ii)]
For $k \in 2\N + 1$, we consider SNLW with a damping term.
More precisely, we study
the following stochastic damped nonlinear wave equation (SdNLW):
\begin{align}
\dt^2 u + \dt u + (1 -  \Dl)  u   + u^k = \sqrt 2\xi.
\label{SNLW1a}
\end{align}

\noi
This equation is known as the hyperbolic counterpart
of the stochastic quantization equation studied in the parabolic setting
\cite{DPD}.
By exploiting (formal) invariance of the Gibbs measure for the dynamics, 
we prove almost sure global well-posedness of SdNLW \eqref{SNLW1a}.
See Subsection \ref{SUBSEC:hyp}.

\end{itemize}

\medskip

The main difficulty in studying these problems, even locally in time,  
 comes
from the roughness of the space-time white noise.
The stochastic convolution $\Psi$, 
solving the linear stochastic wave equation:
\begin{equation}
 \dt^2 \Psi + (1 - \Dl)\Psi   =  \xi, 
\label{SNLW1b}
\end{equation}

\noi
is not a classical function but is merely a  distribution
for the spatial dimension $d \geq 2$.
In particular, 
there is  an issue in  making sense of powers $\Psi^k$ and,  
consequently,  of the full nonlinearity 
 $u^k$ in \eqref{SNLW1}. 
This requires us to modify the equation 
by introducing  a proper  renormalization.
In fact, 
for the models \eqref{SNLW1} and \eqref{SNLW1a} without renormalization, 
a phenomenon of triviality is known to hold
\cite{russo4, OOR}; 
roughly speaking, 
  extreme oscillations 
  make 
solutions to~\eqref{SNLW1} (or~\eqref{SNLW1a}) with regularized noises
tend to that to the  linear stochastic wave equation~\eqref{SNLW1b}
(or the trivial solution) as the regularization is removed.

In the following, let us briefly go over the local well-posedness argument in \cite{GKO1}
and introduce a renormalized equation.
See also \cite{OT2}.
We first express the stochastic convolution
(with the zero initial data)
in terms of a stochastic integral.
With  $\jb{\,\cdot\,} = (1+|\cdot|^2)^\frac{1}{2}$, 
let $S(t)$ denote 
the linear wave propagator:
\begin{equation}\label{lin1}
S(t) = \frac{\sin(t\jb{\nabla})}{\jb{\nabla}},
\end{equation} 

\noi
defined as a Fourier multiplier operator.
Namely, 
we set\footnote{Hereafter, we drop the harmless factor $2\pi$.} 
\[ 
S(t) f  = \sum_{n \in \Z^2 }
\frac{\sin(t\jb{n})}{\jb{n}}
\ft f (n) e_n, \]

\noi
where
$\ft{f}(n)$ is the Fourier coefficient of $f$
and
$e_n(x)= e^{in\cdot x}$.
Then, the stochastic convolution  $\Psi$, 
solving \eqref{SNLW1b}, is given by 
\begin{align} 
\Psi (t) 
 = \int_0^tS(t - t')dW(t'), 
\label{PsiN}
\end{align}

\noi
where  $W$ denotes a cylindrical Wiener process on $L^2(\T^2)$:
\begin{align}
W(t)
: = \sum_{n \in \Z^2 } B_n (t) e_n
\label{W1}
\end{align}

\noi
and  
$\{ B_n \}_{n \in \Z^2}$ 
is defined by 
$B_n(t) = \jb{\xi, \ind_{[0, t]} \cdot e_n}_{ x, t}$.
Here, $\jb{\cdot, \cdot}_{x, t}$ denotes 
the duality pairing on $\T^2\times \R$.
As a result, 
we see that $\{ B_n \}_{n \in \Z^2}$ is a family of mutually independent complex-valued\footnote
{In particular, $B_0$ is  a standard real-valued Brownian motion.} 
Brownian motions conditioned so that $B_{-n} = \cj{B_n}$, $n \in \Z^2$. 
By convention, we normalized $B_n$ such that $\text{Var}(B_n(t)) = t$.

Given $N \in \N$, 
we define the truncated stochastic convolution $\Psi_N = \P_N \Psi$, 
solving the truncated linear stochastic  wave equation:
\begin{align*}
\dt^2 \Psi_N + (1-\Dl)\Psi_N = \P_N \xi
\end{align*}

\noi
with the zero initial data.
Here,   $\P_N$ denotes the frequency cutoff onto 
the spatial frequencies $\{|n| \leq N\}$.
Then, for each fixed $x \in \T^2$ and $t \geq 0$,  
we see that  $\Psi_N(x, t)$
 is a mean-zero real-valued Gaussian random variable with variance
\begin{align}
\begin{split}
\s_N(t) &  \stackrel{\text{def}}{=} \E \big[\Psi_N(x, t)^2\big]
 =  \sum_{\substack{n \in \Z^2\\|n|\leq N}} 
\int_0^t \bigg[\frac{\sin((t - t')\jb{n})}{\jb{n}} \bigg]^2 dt'
\\
& =  \sum_{\substack{n \in \Z^2\\|n|\leq N}}
\bigg\{ \frac{t}{2 \jb{n}^2} - \frac{\sin (2 t \jb{n})}{4\jb{n}^3 }\bigg\}\sim  t \log N
\end{split}
\label{sig}
\end{align}

\noi
\noi
for  $N \gg 1$.
We point out that the variance $\s_N(t)$ is time dependent.
For any $t > 0$, 
we see that  $\s_N(t) \to \infty$ as $N \to \infty$, 
which can be used to show  that $\{\Psi_N(t)\}_{N \in \N}$
is almost surely unbounded in $W^{0, p}(\T^2)$ for any $1 \leq p \leq \infty$.

Let $u_N$ denote the solution to SNLW \eqref{SNLW1}
with the regularized noise $\P_N\xi$.
Proceeding with the following decomposition of $u_N$ (\cite{McKean, BO96, DPD}):
\begin{equation}
\label{A1}
u_N = \Psi_N + v_N.
\end{equation}

\noi
Then, we see that 
 the residual term $v_N$ satisfies 
 \begin{equation}
\dt^2 v_N + (1 - \Dl) v_N +  \sum_{\ell=0}^k {k\choose \ell} \Psi_N^\ell v_N^{k-\ell} = 0.
\label{SNLW2}
\end{equation}

\noi
Note that, 
due to the deficiency of regularity, 
the power $\Psi_N^\ell$ does not converge to any limit  as $N\to \infty$. 
This is where we  introduce
the Wick renormalization.
Namely, 
we replace  $\Psi^\ell_N$
by  its Wick ordered counterpart: 
\begin{align}
:\!\Psi^\l_N(x, t) \!: \, \stackrel{\text{def}}{=} H_\ell(\Psi_N(x, t);\sigma_N(t)), 
\label{Herm1}
\end{align}

\noi
where 
$H_\l(x; \s )$ is the Hermite polynomial of degree $\l$ with variance parameter $\s$.
See Section~\ref{SEC:2}.
Then, for each $\l \in \N$, the Wick power
$ :\! \Psi_N^\l \!:$
converges 
to a limit, denoted by 
$:\! \Psi^\l  \!: \,$, 
in $C([0,T];W^{-\eps,\infty}(\T^2))$
for any $\eps > 0$ and $T > 0$, 
almost surely (and also in $L^p(\O)$
for any $p < \infty$).
See Lemma \ref{LEM:Psi} below.
This Wick renormalization gives rise 
to the renormalized version of \eqref{SNLW2}:
 \begin{equation*}
\dt^2 v_N + (1 - \Dl) v_N +  \sum_{\ell=0}^k {k\choose \ell} :\!\Psi_N^\l\!:  v_N^{k-\ell} = 0.
\end{equation*}

\noi
By taking a limit as $N \to \infty$, 
we then obtain the limiting equation:
 \begin{equation}
\dt^2 v  + (1 - \Dl) v +  \sum_{\ell=0}^k {k\choose \ell} :\!\Psi^\l\!:  v^{k-\ell} = 0.
\label{SNLW4}
\end{equation}

\noi
Given the almost sure space-time regularity of the Wick powers
$:\!\Psi^\l\!:$, $\l = 1, \dots, k$, standard deterministic analysis with the Strichartz estimates
and the product estimates (Lemma~\ref{LEM:bilin})
yields local well-posedness of \eqref{SNLW4}
(for $v$).
Recalling the decomposition~\eqref{A1}, 
this argument also shows that 
the solution  $u_N = \Psi_N + v_N $ to the renormalized 
SNLW with the regularized noise $\P_N \xi$:
\[ \dt^2 u_N + (1 -  \Dl)  u_N   + :\!u_N^k\!: \, = \P_N \xi, \]

\noi
where the renormalized nonlinearity $:\!u_N^k\!: $ is interpreted as 
\[ :\!u_N^k \!: \, 
= \, :\!(\Psi_N + v_N)^k \!: \, 
= 
 \sum_{\ell=0}^k {k\choose \ell} :\!\Psi_N^\l\!:  v_N^{k-\ell} , 
\]

\noi
converges almost surely to  a stochastic 
process $u
 = \Psi + v$, 
where $v$ satisfies \eqref{SNLW4}.
It is in this sense that we say that 
the renormalized SNLW:
\begin{align*}
\dt^2 u + (1 -  \Dl)  u   +  :\!u^k\!: \,= \xi
\end{align*}

\noi
is locally well-posed
(for  initial data of suitable regularity).

\begin{remark}\label{REM:lin}\rm

The equation \eqref{SNLW1}
is also known as 
 the stochastic nonlinear Klein-Gordon equation.
In the following, however, we simply refer to \eqref{SNLW1} as 
the stochastic nonlinear wave equation.

In \cite{GKO1}, 
we treated the equation \eqref{SNLW1}
with the mass-less linear part  $\dt^2 u- \Dl u$.
Note that the same results in \cite{GKO1} with  inessential modifications
also hold for \eqref{SNLW1}
with the massive linear part  $\dt^2 u + (1- \Dl) u$.
Conversely, Theorem \ref{THM:GWP1} below
also holds for SNLW 
with the mass-less linear part  $\dt^2 u- \Dl u$.
We point out, however, that 
for our second main result (Theorem \ref{THM:GWP2}), 
we need to work with the massive linear part
in order to avoid a problem at the zeroth frequency in 
the Gibbs measure construction; see \cite{OT1}.
For this reason, we work with  the massive case in this paper.
\end{remark}

\subsection{Global well-posedness of the cubic SNLW}
\label{SUBSEC:GWP0}

Our first goal is to construct global-in-time dynamics
for the renormalized cubic SNLW.
In the following, we study \eqref{SNLW4} with $k = 3$:
 \begin{equation}
\dt^2 v  + (1 - \Dl) v +  
v^3 + 3 v^2 \Psi +  3v :\!\Psi^2\!: + :\!\Psi^3\!: \, = 0.
\label{SNLW6}
\end{equation}

\noi
In \cite{GKO1}, 
it was shown that \eqref{SNLW6}
is locally well-posed
in $\H^s(\T^2)  \stackrel{\text{def}}{=} H^s(\T^2)\times H^{s-1}(\T^2)$
for $s > \frac 14$.
Furthermore, the following blowup alternative holds almost surely;
either the solution $v$ exists globally in time
or there exists some finite time $T_* = T_*(\o) > 0$
such that  
\begin{align}
\lim_{t \nearrow T_*}\| \vec v(t) \|_{\H^\s} = \infty,
\label{BA}
\end{align}

\noi
where $\vec v = (v, \dt v)$
and $\s = \min(s, 1- \eps)$ for any small $\eps > 0$.
While the blowup alternative~\eqref{BA} is not explicitly proven in \cite{GKO1}, 
it easily follows as a consequence of the (deterministic) contraction 
argument used to study \eqref{SNLW6} in \cite{GKO1}.

In the parabolic setting, 
there are recent works \cite{MW1, MW2, GH, MoinatW}
on global well-posedness of the parabolic $\Phi^4_d$-model
via deterministic approaches.
The main ingredient  in \cite{MW1, MW2} is  
a (non-trivial) adaptation of a standard globalization argument
for a nonlinear heat equation
by   controlling  the (weighted) $L^p$-norm
of the smoother part of a solution (corresponding to $v$ in \eqref{SNLW6}).
Due to a weaker smoothing property, however, 
the situation is much more involved
in the case of the wave equation.

Essentially speaking, the only known way
to prove global well-posedness 
for the deterministic cubic nonlinear wave equation (NLW):
 \begin{equation}
\dt^2 v  + (1 - \Dl) v +  
v^3  = 0
\label{NLW1}
\end{equation}
 
\noi
 (except in the small data\footnote{This includes
the construction of solutions near a particular solution.}
regime)
is to exploit the 
energy $E(\vec v)$ given by 
\begin{align}
E(\vec v ) = \frac{1}{2}\int_{\T^2}\big( v^2 +  |\nb v|^2\big) dx
+ 
\frac{1}{2}\int_{\T^2} (\dt v)^2dx
+ \frac1{4} \int_{\T^2} v^{4} dx, 
\label{Hamil}
\end{align}

\noi
which is conserved for smooth solutions.
There are two sources of difficulty
in proving global well-posedness of the cubic SNLW \eqref{SNLW6}.

\smallskip

\begin{itemize}
\item[(i)] The first problem comes from the lack of regularity
of the solution $\vec v = (v, \dt v)$
to~\eqref{SNLW6}.
Due to the roughness of the stochastic convolution $\Psi$, 
we easily see that 
$\vec v  \in C([0, T]; \H^s(\T^2)) $
only 
for $s < 1$.
Namely, for a solution $\vec v$ to \eqref{SNLW6}, 
the energy $E(\vec v)$ is infinite.
In order to overcome this difficulty, we
propose to use the $I$-method 
introduced by Colliander-Keel-Staffilani-Takaoka-Tao \cite{CKSTT1, CKSTT2}.
See for example
\cite{Roy} on an application of the $I$-method to the cubic NLW \eqref{NLW1} on $\T^2$
in the deterministic setting.

\smallskip

\item[(ii)]
The second problem comes
from the fact that $v$ satisfies 
the cubic NLW  with perturbations, 
which results in the non-conservation of the energy $E(\vec v)$
even if $\vec v(t)$ were in $\H^1(\T^2)$.

\end{itemize}

\smallskip

\noi
The second problem could be easily remedied 
if $\Psi$ were slightly  smoother.
Given  $\Psi \in C(\R_+; L^\infty(\T^2))$, 
consider 
 \begin{equation}
\dt^2 v  + (1 - \Dl) v +  
v^3 + 3 v^2 \Psi +  3v \Psi^2  + \Psi^3 = 0.
\label{SNLW7}
\end{equation}

\noi
In this case, we can apply  the globalization argument  by Burq-Tzvetkov~\cite{BT2},
originally introduced  in the context
of the cubic NLW on $\T^3$ with random initial data.
Namely, by Cauchy-Schwarz inequality
along with \eqref{SNLW7} and Young's inequality, we have 
\begin{align*}
|\dt E(\vec v) |
& = \bigg|\int_{\T^2}
(\dt v )\big\{ \dt^2 v+ (1 - \Dl) v + v^3\big\} dx\bigg|\\
& \leq \big(E(\vec v)\big)^\frac{1}{2}
\bigg( \|\Psi\|_{C_TL^\infty_x}^2 \int_{\T^2}v^4 dx + \|\Psi\|_{C_T L^6_x}^6\bigg)^\frac{1}{2}\\
& \leq C(T, \Psi) \big( 1 + E(\vec v)\big) 
\end{align*}

\noi
for any given $T > 0$, 
where $C_TL^p_x = C([0, T]; L^p(\T^2))$.
Then, global well-posedness of~\eqref{SNLW7} in $\H^1(\T^2)$
follows from Gronwall's inequality.

As  described above, 
we can handle each of the difficulties (i) and (ii) by a standard approach if it occurs one at a time.
The main difficulty in proving
global well-posedness of the cubic SNLW
\eqref{SNLW6}
lies in the fact that we need to handle the difficulties
(i) and~(ii) {\it at the same time}.
This combination of the problems (i) and (ii) makes the problem significantly harder.

We now state our first main result.

\begin{theorem}\label{THM:GWP1}
Let $s > \frac 45$.
Then, the renormalized cubic SNLW \eqref{SNLW6} on $\T^2$
is globally well-posed in $\H^s(\T^2)$.
More precisely, 
given $(\phi_0, \phi_1) \in \H^s(\T^2)$, 
the solution $v$ to \eqref{SNLW6} 
exists globally in time, almost surely, 
such that 
$(v, \dt v) \in C(\R_+; \H^\s(\T^2))$, 
where
$\s = \min(s, 1- \eps)$ for any small $\eps > 0$.
\end{theorem}

For simplicity, we only consider the case
$\frac 45 < s < 1$ such that $\s = s$.
The main approach is to combine the $I$-method with 
a Gronwall-type argument.
Let us first recall the main idea of the $I$-method.
Fix $0 < s < 1$.
Given $N \geq 1$, we define a 
smooth, radially symmetric, non-increasing (in $|\xi|$)
multiplier $m_N \in C^\infty(\R^2; [0, 1])$, satisfying 
\begin{equation}
m_N(\xi)=
\begin{cases}
1, & \text{if } |\xi|\le N, \\
\Big(\frac{N}{|\xi|}\Big)^{1-s}, & \text{if } |\xi|\ge2N.
\end{cases}
\label{I0a}
\end{equation} 

\noi
We then define the $I$-operator $I = I_N$
to be the Fourier multiplier operator with the multiplier $m_N$:
\begin{align}
\ft{I_Nf}(n)=m_N(n)\ft{f}(n).
\label{I0}
\end{align}

\noi
Then, we see that 
$I_N$ acts as the \underline{\it i}\hspace{0.6pt}dentity operator on low frequencies $\{ |n| \leq N\}$, 
while it acts as a fractional  
\underline{\it i}\hspace{0.6pt}ntegration operator of order $1-s$ on high frequencies $\{ |n| \geq 2N\}$.
From the definition, it is easy to see that 
 $I f \in H^1(\T^2)$ if and only if $f \in H^s(\T^2)$
with the bound:
\begin{align}
\|f\|_{H^s}\les \|I f\|_{ H^1} \les N^{1-s} \|f\|_{H^s}.
\label{I1}
\end{align}

\noi
Moreover, by the Littlewood-Paley theory, we have\footnote{Here, $W^{s, r}(\T^2)$
denotes the usual $L^r$-based Sobolev space (Bessel potential space) defined by the norm:
\[\| u\|_{W^{s, r}} = \| \jb{\nb}^s u \|_{L^r} = \big\| \F^{-1}( \jb{n}^s \ft u(n))\big\|_{L^r}.\]

\noi
When $r = 2$, we have $H^s(\T^2) = W^{s, 2}(\T^2)$.} 
\begin{align}
\|I f\|_{W^{s_0 + s_1, p}}\les  N^{s_1}\|f\|_{W^{s_0, p}}
\label{I2}
\end{align}

\noi
for any $s_0 \in \R$, $0 \leq s_1 \leq 1- s$,  and $1 < p < \infty$.

Let $ \frac 45 < s < 1$.
Given initial data $(\phi_0. \phi_1) \in \H^s(\T^2)$, 
we consider the $I$-SNLW:
 \begin{equation}
\dt^2 I v  + (1 - \Dl) I v +  
I (v^3) + 3 I(v^2 \Psi) +  3I(v :\!\Psi^2\!:) + I(:\!\Psi^3\!: ) = 0.
\label{SNLW8}
\end{equation}

\noi
The local well-posedness of \eqref{SNLW6} implies
local well-posedness of the $I$-SNLW \eqref{SNLW8}.
In view of the blowup alternative \eqref{BA}
and \eqref{I1}, 
our main task is to control the growth of 
the modified energy $E(I\vec v)$.
Note that there are two sources
for the non-conservation of the modified energy 
$E(I\vec v)$, reflecting the problems (i) and (ii) discussed above:
(i) The main part of the nonlinearity is 
$I(v^3)$,  not the cubic power $(Iv)^3$, 
and (ii) 
there are perturbative terms:
$3 I(v^2 \Psi) +  3I(v :\!\Psi^2\!:) + I(:\!\Psi^3\!: )$.
Indeed, a direct computation with \eqref{Hamil} and~\eqref{SNLW8}
gives
\begin{align}
 E(I\vec v)(t) -  E(I\vec v)(0) 
&  = \int_0^t \int_{\T^2} (\dt I v) \big\{ - I (v^3) + (Iv)^3\big\} dx dt \notag \\
& \hphantom{X}
-3 
\int_0^t \int_{\T^2} (\dt I v) I(v^2 \Psi) dx dt \notag \\
& \hphantom{X}
-3
\int_0^t \int_{\T^2} (\dt I v) I(v :\!\Psi^2\!:) dx dt \notag \\
& \hphantom{X}
- 
\int_0^t \int_{\T^2} (\dt I v) I(:\!\Psi^3\!: ) dx dt \notag \\
& = : A_1 + A_2 + A_3 + A_4.
\label{E1}
\end{align}

\noi
The first term $A_1$ represents
the main commutator part, resulting from the application of
the $I$-operator, 
and we estimate this part by establishing a certain commutator estimate
(as in the deterministic setting).
On the other hand, 
 the second, third, and fourth terms
$A_2$, $A_3$, and $A_4$ represent the contributions
from the perturbative terms in \eqref{SNLW8},
which are to be controlled by a Gronwall-type argument as above.\footnote{As
we see in Section \ref{SEC:GWP1}, 
these terms also contain the commutator parts as well.
For simplicity, we ignore this issue in this part of discussion.}
The worst contribution comes from~$A_2$.
In order to control this term, 
the standard estimate 
\eqref{I2} with the fact that $\Psi(t) \in W^{-\eps, \infty}(\T^2)$
is too crude since it loses a positive power of $N$.  We instead need
to use a finer regularity property of $\Psi(t)$, 
namely, it is logarithmically divergent from $L^p(\T^2)$.
See Lemma \ref{LEM:log} below.
At the end of the day, we end up 
with a Gronwall-type estimate, 
where the right-hand side has a logarithmically superlinear growth.
Roughly speaking, we obtain an estimate of the form:
\begin{align}
\dt  E(I\vec v)(t) \les 
 E(I\vec v)(t) \log\big( E(I\vec v)(t)\big).
 \label{E1a}
\end{align}

\noi
See \eqref{Y1} and \eqref{K15} below
for precise bounds.
We then implement an iterative argument, 
proceeding over time intervals  of fixed size, 
by {\it choosing an increasing sequence of 
the parameters $N_k$ for the $I$-operator}. 
See Subsection \ref{SUBSEC:3.2} for details.

\begin{remark}\rm 
(i) In a standard application of the $I$-method, one
first fixes the large target time $T \gg 1$
and then chooses a parameter $N = N(T) \gg1 $.
For our problem, this is not sufficient.
We instead need to choose 
 an increasing  sequence of 
the parameters $N_k$ for the $I$-operator
over different local-in-time intervals.
It would be of interest to investigate
a possible application of
this  new type of the $I$-method argument 
in the deterministic or random data setting
(other than that mentioned in the following remark).

\smallskip

\noi
(ii) 
A standard application of the $I$-method
 yields a polynomial (in time) growth bound on the Sobolev norm
of a solution.  See,  for example,  Section 6 in \cite{CKSTT2}.
A close examination of 
our hybrid
argument yields
a double exponential growth bound on the $\H^s$-norm of a solution.
 See Remark \ref{REM:bound}.
We point out that 
such a 
double exponential  bound would follow as a direct consequence of 
the estimate\footnote{Note that we do not quite obtain the estimate \eqref{E1a}
for the modified energy $ E(I\vec v)$.
See \eqref{Y1} and~\eqref{K15} below
for the actual bounds.}
 \eqref{E1a}.
 While it may be possible to improve this 
double exponential  bound, 
 we do not know how to do so at this point.
Such an argument 
would require a new globalization approach.
Lastly, 
we note that, 
while one may expect a subpolynomial growth in the deterministic setting, 
we  expect at best a polynomial growth bound
for  the (undamped) SNLW
due to the polynomial growth (in time)
 of the  stochastic convolution
 (which is essentially a Brownian motion in time).
Compare this with the damped case, 
where the invariant measure argument yields
a logarithmic growth bound;
see 
Remark \ref{REM:G2}.

\end{remark}

\begin{remark}\rm
In \cite{OT2}, 
Thomann and the third author proved
 almost sure global well-posedness of the renormalized defocusing 
cubic NLW on $\T^2$ with 
the random data distributed by 
the massive Gaussian free field.
The proof in \cite{OT2} was based on  (formal) invariance
of the Gibbs measure and Bourgain's invariant measure argument.
We point out that 
a slight modification of the proof of Theorem \ref{THM:GWP1}
provides another proof of this almost sure global well-posedness result
via a pathwise argument
(without using the invariant measure argument).

\end{remark}

\begin{remark}\label{REM:I3}\rm 
(i) Theorem \ref{THM:GWP1} establishes global well-posedness
of the renormalized cubic SNLW \eqref{SNLW6} on $\T^2$
in $\H^s(\T^2)$ for $s > \frac 45$, 
which leaves a gap to the local well-posedness threshold $s > \frac 14$ from \cite{GKO1}.
It may be possible to refine the $I$-method part (for example, by using analysis from \cite{Roy})
to lower regularities (to some extent).
We, however, decided not to pursue this issue since 
our globalization argument presented in Section \ref{SEC:GWP1}
is already quite involved, 
and our main goal in this part is to present this hybrid argument
of the $I$-method with a Gronwall-type argument
in its simplest form.

\smallskip

\noi
(ii) In a recent work 
\cite{Tolomeo2}, the fourth author extended Theorem \ref{THM:GWP1}
to the renormalized cubic SNLW  on $\R^2$.
For this problem, one needs to handle
not only the roughness of the noise but also its unboundedness.

\smallskip

\noi
(iii) 
In \cite{Forlano}, Forlano recently adapted our globalization argument
in studying 
the BBM equation with random initial data outside $L^2(\T)$.

\smallskip

\noi
(iv)  At this point, we do not know how to prove global well-posedness
of the renormalized  (undamped) SNLW with (super-)quintic nonlinearity.
Even with a smoother noise, 
one would need to use a trick 
introduced in 
\cite{OP1} to handle the high homogeneity.
See for example   \cite{MPTW}
for global well-posedness of the stochastic nonlinear beam equations
on $\T^3$.
\end{remark}

\begin{remark}\rm
In order to prove global well-posedness
of a stochastic PDE, 
we employ the $I$-method
to study the equation \eqref{SNLW6} 
for $v = u - \Psi$.
As such, our argument is essentially pathwise
and thus entirely deterministic, once we have 
a control on the relevant stochastic terms.

In a recent work \cite{CLO}, 
the third author with Cheung and Li
implemented the $I$-method
to prove global well-posedness of
stochastic nonlinear Schr\"odinger equations (SNLS)
below the energy space.
In estimating the growth of the modified energy, 
the authors used
Ito's lemma, which led to a careful stopping time argument
(rather than a usual application of the $I$-method, 
where one iterates a local-in-time argument
with a control on the modified energy).
The argument introduced  in \cite{CLO}
is  a 
natural\footnote{In particular, 
the authors in \cite{CLO} studied the growth of the modified energy 
of a solution $u$ (rather than the residual part $v = u - \Psi$)
via Ito's lemma, 
which is a natural extension of the 
$H^1$-global well-posedness result  on SNLS by 
 de Bouard and Debussche \cite{DD2} to the low regularity setting.}
extension of the $I$-method to the stochastic setting, which can be applied
to a wide class of stochastic dispersive equations.

\end{remark}

\subsection{Hyperbolic $\Phi_2$-model
and the Gibbs measure}
\label{SUBSEC:hyp}

In this subsection,  we consider 
the following stochastic damped nonlinear wave equation (SdNLW):
\begin{align}
\dt^2 u + \dt u + (1 -  \Dl)  u   + u^k = \sqrt 2\xi
\label{SNLW9}
\end{align}

\noi
for $k \in 2 \N + 1$.
This model is known as  the 
so-called canonical stochastic quantization equation
for the $\Phi^{k+1}_2$-model \cite{RSS}; see also a discussion below.
The local well-posedness argument from \cite{GKO1}
for the undamped (renormalized) SNLW 
is readily applicable 
to yield local well-posedness of 
 (the renormalized version of) SdNLW~\eqref{SNLW9}
 for any $k \in 2 \N + 1$.
Moreover, when $k = 3$, 
a slight modification of the proof of Theorem~\ref{THM:GWP1}
provides a deterministic (i.e.~pathwise) argument, 
establishing global well-posedness in the damped case.
As pointed out in Remark \ref{REM:I3}, 
such a deterministic argument is limited to $k = 3$ at this point.
In the damped case, 
 however, 
we can rely on a probabilistic argument
in order to construct global-in-times dynamics
for \eqref{SNLW9} with general  $k \in 2 \N + 1$.
More precisely, 
we construct global-in-time dynamics
for~\eqref{SNLW9},  
by exploiting (formal) invariance of 
 the 
 Gibbs measure with the density: 
\begin{align}
``d\rhoo(u,\dt u ) = Z^{-1}e^{-E(u,\dt u )}dud(\dt u)", 
\label{Gibbs1}
\end{align}

\noi
where $E(u, \dt u)$ denotes the energy (= Hamiltonian):
\begin{align}
E(u, \dt u ) = \frac{1}{2}\int_{\T^2}\big( u^2 +  |\nb u|^2\big) dx
+ 
\frac{1}{2}\int_{\T^2} (\dt u)^2dx
+ \frac1{k+1} \int_{\T^2} u^{k+1} dx
\label{Hamil2}
\end{align}

\noi
for the (deterministic undamped) NLW:
\begin{align}
\dt^2 u +  (1 -  \Dl)  u   + u^k = 0.
\label{NLW2}
\end{align}

\noi
By drawing an analogy to finite-dimensional Hamiltonian systems, 
the Gibbs measure $\rhoo$ was  expected to be invariant
under the NLW dynamics \eqref{NLW2}.
In  \cite{OT2}, 
the third author
and  Thomann showed that this is  indeed the case.
As for SdNLW \eqref{SNLW9}, 
we can view it as the superposition of 
the NLW dynamics \eqref{NLW2}
and the 
Ornstein-Uhlenbeck dynamics 
(for the  component $\dt u$):
\begin{align*}
\dt (\dt u) = - \dt u + 
\sqrt 2 dW, 
\end{align*}

\noi
each of which preserves   the Gibbs measure~$\rhoo$ in \eqref{Gibbs1}.
Hence, we expect 
the Gibbs measure~$\rhoo$ to be 
invariant under the dynamics of SdNLW \eqref{SNLW9}. 

By substituting \eqref{Hamil2} in the exponent
of \eqref{Gibbs1}, we see that the Gibbs measure $\rhoo$ decouples
into the $\Phi^{k+1}_2$-measure on $u$
and the white noise measure on $\dt u$.
The dynamical model \eqref{SNLW9} then corresponds to the 
canonical stochastic quantization equation\footnote{Namely, the Langevin equation
with the momentum $v = \dt u$.}
of the $\Phi^{k+1}_2$-model; see~\cite{RSS}. 
For this reason, 
we also refer to \eqref{SNLW9}
as the hyperbolic $\Phi^{k+1}_2$-model.

In order to make our discussion rigorous, 
let us introduce some notations.
Given $ s \in \R$, 
let $\mu_s$ denote
a Gaussian measure on periodic distributions,   formally defined by
\begin{align}
 d \mu_s 
   = Z_s^{-1} e^{-\frac 12 \| u\|_{{H}^{s}}^2} du
& =  Z_s^{-1} \prod_{n \in \Z^2} 
 e^{-\frac 12 \jb{n}^{2s} |\ft u(n)|^2}   
 d\ft u(n) .
\label{gauss0}
\end{align}

\noi
Note that $\mu_1$ corresponds to the massive Gaussian free field, 
while $\mu_0$ corresponds to the white noise.
We set 
\begin{align}
\muu_s = \mu_s \otimes \mu_{s-1} .
\label{gauss1}
\end{align}

\noi
In particular, when $s = 1$, 
 the measure $\muu_1$ is defined as 
   the induced probability measure
under the map:
\begin{equation*}
\o \in \O \longmapsto (u^1(\o), u^2(\o)),
 \end{equation*}

\noi
where $u^1(\o)$ and $u^2(\o)$ are given by
\begin{equation}\label{series}
u^1(\o) = \sum_{n \in \Z^2} \frac{g_n(\o)}{\jb{n}}e_n
\qquad\text{and}\qquad
u^2(\o) = \sum_{n \in \Z^2} h_n(\o)e_n.
\end{equation}

\noi
Here, 
   $\{g_n,h_n\}_{n\in\Z^2}$ denotes  a family of independent standard complex-valued  Gaussian random variables conditioned so that $\cj{g_n}=g_{-n}$ and $\cj{h_n}=h_{-n}$, 
 $n \in \Z^2$.
It is easy to see that $\muu_1 = \mu_1\otimes\mu_0$ is supported on
$\H^{s}(\T^2)$
for $s < 0$ but not for $s \geq 0$.

With  
 \eqref{Hamil2}, 
 \eqref{gauss0}, and 
 \eqref{gauss1},  
 we can formally write the Gibbs measure $\rhoo$ in \eqref{Gibbs1} as 
\begin{align}\label{Gibbs2}
d\rhoo(u,\dt u ) \sim e^{
- \frac1{k+1} \int_{\T^2} u^{k+1} dx} 
d\muu_1 (u, \dt u ).
\end{align} 

\noi
In view of the roughness of the support of $\muu_1$, 
the nonlinear term $\int_{\T^2} u^{k+1} dx $ in \eqref{Gibbs2} is not well defined
and thus a proper renormalization is required to give a meaning to \eqref{Gibbs2}.

Given a random variable $X$, 
let $\L(X)$ denote the law of $X$.
Suppose that $\L(u) = \mu_1$.
Then, given $N \in \N$, we have 
 \begin{align}\label{sN}
 \al_N \deff  \E\big[\big(\P_Nu(x)\big)^2\big] 
 =\sum_{\substack{n\in\Z^2\\|n|\leq N}}\frac{1}{\jb{n}^2}
\sim \log N 
 \end{align}

\noi
for $N \gg 1$, 
independent of $x\in\T^2$.
Given $N \in \N$, 
define 
 the truncated renormalized density:
\begin{align}
R_N(u) = 
\exp\bigg( - \frac{1}{k+1} 
\int_{\T^2} :\!(\P_N u)^{k+1}(x) \!: dx\bigg), 
\label{R1}
\end{align}

\noi
where the Wick power $:\!(\P_N u)^{k+1}(x) \!:$ is defined by  
\begin{align*}
:\!(\P_N u)^{k+1}(x) \!:
 \, \deff  H_{k+1} (\P_N u(x); \al_N).
\end{align*}

\noi
Then, it is known that $\{R_N \}_{N \in \N}$
forms a Cauchy sequence in $L^p(\mu_1)$
for any finite $p \geq 1$.
Thus, there exists a random variable 
 $R(u)$
 such that 
 \begin{equation}\label{exp2}
\lim_{N\rightarrow\infty} R_N(u)=R(u)
\qquad \text{in } L^p(\mu_1).
\end{equation}

\noi
See \cite{Simon, GJ, DPT1, OT1} for details.
In view of \eqref{R1} and \eqref{exp2}, 
we can write the limit  as
\begin{align*}
R(u) = 
\exp\bigg( - \frac{1}{k+1} 
\int_{\T^2} :\! u^{k+1}(x) \!: dx\bigg).
\end{align*}

\noi
By defining  the renormalized  truncated Gibbs measure:
\begin{align}\label{GibbsN}
d\rhoo_N(u,\dt u )= Z_N^{-1}R_N(u)d\muu_1(u,\dt u), 
\end{align}

\noi
we then conclude that 
the renormalized truncated  Gibbs measure $\rhoo_N$  converges, in the sense of \eqref{exp2}, 
 to the renormalized Gibbs measure $\rhoo$ given by
\begin{align}
\begin{split}
d\rhoo(u,\dt u )
& = Z^{-1} R(u)d\muu_1(u, \dt u )\\
& = Z^{-1} \exp\bigg( - \frac{1}{k+1} 
\int_{\T^2} :\! u^{k+1}(x) \!: dx\bigg)d\muu_1(u, \dt u ).
\end{split}
\label{Gibbs3}
\end{align}

\noi
Furthermore, 
the resulting Gibbs measure $\rhoo$ is equivalent\footnote{Namely, 
$\rhoo$ and $\muu_1$ are mutually absolutely continuous.} 
to the Gaussian measure $\muu_1$.

Next, 
we move onto the well-posedness theory 
of the hyperbolic $\Phi^{k+1}_2$-model \eqref{SNLW9}.
Let us first introduce  the following  renormalized truncated SdNLW:
\begin{align}
\dt^2 u_N   + \dt u_N  +(1-\Dl)  u_N 
+
\P_N\big(:\!(\P_N u)^{k} \!: \big) 
   = \sqrt{2} \xi 
\label{SNLW10}
\end{align} 

\noi
and its formal limit: 
\begin{align}
\dt^2 u   + \dt u  +(1-\Dl)  u 
+
:\!u^{k} \!: \, 
   = \sqrt{2} \xi  .
\label{SNLW11}
\end{align}

\noi
It is easy to check that the 
 renormalized  truncated Gibbs measure $\rhoo_N$
 is invariant under the truncated dynamics \eqref{SNLW10}.
See Section \ref{SEC:GWP2}.

We now state our second result.

 \begin{theorem}\label{THM:GWP2}
The renormalized SdNLW~\eqref{SNLW11} is almost surely globally well-posed with respect to the renormalized Gibbs measure~$\rhoo$ in~\eqref{Gibbs3}. Furthermore, the renormalized Gibbs measure $\rhoo$ is invariant under the dynamics.

More precisely, there exists a non-trivial stochastic process $(u,\dt u)\in C(\R_+;\H^{-\eps}(\T^2))$ for any $\eps>0$ such that, given any $T>0$, the solution $(u_N,\dt u_N)$ to 
the renormalized truncated SdNLW~\eqref{SNLW10} with 
the random initial data  
$(u_N, \dt u _N)|_{t = 0}$,  distributed according 
to the renormalized  truncated Gibbs measure $\rhoo_N$  in~\eqref{GibbsN}, converges in probability to 
some stochastic process $(u,\dt u)$ in $C([0,T];\H^{-\eps}(\T^2))$.
Moreover,  the law of $(u(t),\dt u(t))$ is given by the renormalized Gibbs measure $\rhoo$ in \eqref{Gibbs3}
 for any $t\ge 0$.
\end{theorem} 

In the context of the renormalized (deterministic) NLW:
\begin{align*}
\dt^2 u    +(1-\Dl)  u 
+:\!u^{k} \!: \, 
   = 0, 
\end{align*} 

\noi
the third author
with Thomann proved an analogous result; see  \cite{OT2}.

In view of the convergence of $\rhoo_N$ to $\rhoo$, 
the invariance of $\rhoo_N$ under 
the truncated SdNLW  dynamics~\eqref{SNLW10}, 
and Bourgain's invariant measure argument \cite{BO94,BO96},
Theorem \ref{THM:GWP2}
follows once we construct the limiting process $(u,\dt u)$ locally in time
with a good approximation property by 
the solution $u_N$ to \eqref{SNLW10}.
Furthermore, in view of the equivalence of 
 $\rhoo$, $\rhoo_N$,  and $\muu_1$, 
 it suffices to study  the renormalized SdNLW 
 \eqref{SNLW10} and \eqref{SNLW11} with the Gaussian 
 random initial data $(\phi_0, \phi_1)$ 
 with $\L(\phi_0, \phi_1) = \muu_1$.

 As in the previous sections,  we proceed with the first order expansion.
  For our damped model, 
we  let $\Phi$  be the solution to the linear stochastic damped wave equation:
\begin{align}\label{SdLW}
\begin{cases}
\dt^2 \Phi + \dt\Phi +(1-\Dl)\Phi  = \sqrt{2}\xi\\
(\Phi,\dt\Phi)|_{t=0}=(\phi_0,\phi_1),
\end{cases}
\end{align}

\noi
where $\L (\phi_0, \phi_1) = \muu_1$. 
Define the linear damped wave propagator $\D(t)$ by 
\begin{equation}
\D(t) = e^{-\frac{t}2}\frac{\sin\Big(t\sqrt{\frac34-\Dl}\Big)}{\sqrt{\frac34-\Dl}}
\label{lin2}
\end{equation} 

\noi
as a Fourier multiplier operator.
Then, the stochastic convolution $\Phi$ can be expressed as 
\begin{align} 
\Phi (t) 
 = \dt\D(t)\phi_0 + \D(t)(\phi_0+\phi_1)+ \sqrt{2}\int_0^t\D(t - t')dW(t'), 
\label{PhiN}
\end{align}

\noi
where  $W$ is as in \eqref{W1}.
A direct computation shows that $\Phi_N(x, t)=\P_N\Phi(x, t)$
 is a mean-zero real-valued Gaussian random variable with variance
\begin{align*}
 \E \big[\Phi_N(x, t)^2\big] = \E\big[\big(\P_N\Phi(x, t)\big)^2\big]
 = \al_N
\end{align*}

\noi
for any $t\ge 0$, $x\in\T^2$,  and $N \ge 1$,
where $\al_N$ is as in \eqref{sN}.
We point out that unlike $\s_N(t)$ in \eqref{sig}, 
the variance $\al_N$ is time independent.
This is due to the fact that the massive Gaussian free field $\mu_1$
is invariant under the dynamics
of   the linear stochastic damped wave equation \eqref{SdLW}.
 
Let  $u_N$ be the solution to \eqref{SNLW10}
 with $\L\big((u_N, \dt u_N)|_{t=0}\big) = \muu_1$.
 Then, by writing $u_N$ 
 as 
\begin{align}
u_N = v_N + \Phi
= (v_N +  \Phi_N) + \P_N^\perp \Phi_N,
\label{decomp2}
\end{align}

\noi
where $\P_N^\perp = \Id - \P_N$, 
we see that the dynamics of 
the   renormalized truncated SdNLW~\eqref{SNLW10}
decouples into 
the linear dynamics for the high frequency part given by $\P_N^\perp \Phi_N$
and 
the nonlinear dynamics for the low frequency part $\P_N u_N$:
\begin{align}
\dt^2 \P_N u_N   + \dt \P_N u_N  +(1-\Dl)  \P_N u_N 
+
\P_N\big(:\!(\P_N u)^{k} \!: \big) 
   = \sqrt{2} \P_N \xi  .
\label{SNLW11a}
\end{align}

\noi
Then,   
 the residual part $v_N = \P_N u_N - \Phi_N $ satisfies the following equation:
 \begin{align}
\begin{cases}
\dt^2 v_N + \dt v_N +(1-\Dl)v_N  +
 \sum_{\ell=0}^k {k\choose \ell} \P_N \big(:\!\Phi_N^\l\!:  v_N^{k-\ell}\big)
=0\\
(v_N,\dt v_N)|_{t = 0}=(0,0), 
\end{cases}
\label{SNLW12}
\end{align}

\noi
where the Wick power is defined by 
\begin{align}
:\!\Phi_N^\l (x, t) \!:
 \, \deff  H_{\l} (\Phi_N(x, t); \al_N).
\label{Herm3}
\end{align}

\noi
As in the undamped case discussed earlier, 
 for each $\l \in \N$, the Wick power
$ :\! \Phi_N^\l \!:$
converges 
to a limit, denoted by 
$:\! \Phi^\l  \!: \,$, 
in $C([0,T];W^{-\eps,\infty}(\T^2))$
for any $\eps > 0$ and $T > 0$, 
almost surely (and also in $L^p(\O)$
for any $p < \infty$).
See Lemma \ref{LEM:Psi} below.
This allows us to formally obtain the limiting equation:
 \begin{align}
\begin{cases}
\dt^2 v + \dt v +(1-\Dl)v  +
 \sum_{\ell=0}^k {k\choose \ell} :\!\Phi^\l\!:  v^{k-\ell}
=0\\
(v,\dt v)|_{t = 0}=(0,0).
\end{cases}
\label{SNLW13}
\end{align}

\noi
Note that the damped wave propagator $\D(t)$ in \eqref{lin2} satisfies
the same Strichartz estimates as the standard wave propagator $S(t)$ in \eqref{lin1}.
Hence,  by following the argument in \cite{GKO1}, 
we can prove local well-posedness of \eqref{SNLW13}, 
using the Strichartz estimates.
In Section \ref{SEC:GWP2}, 
we instead present a simple argument 
for  local well-posedness
of \eqref{SNLW13} based on Sobolev's inequality.
See Proposition \ref{PROP:LWP}.
This local well-posedness can also be applied 
to the truncated equation \eqref{SNLW12}, 
uniformly in $N \in \N$.
Once we prove (uniform in $N$) local well-posedness of~\eqref{SNLW12} and~\eqref{SNLW13}
and check invariance of 
the truncated Gibbs measure $\rhoo_N$ under 
the truncated SdNLW  dynamics~\eqref{SNLW10}, 
the rest of the proof of Theorem \ref{THM:GWP2}
follows from a standard application of Bourgain's invariant measure argument,
whose details we omit.
See, for example,~\cite{ORTz} for further details,
where Robert, Tzvetkov, and the third author
extended Theorem \ref{THM:GWP2}
to the case of two-dimensional compact Riemannian manifolds
without boundary.

\begin{remark}\label{REM:G2}\rm
(i) In Section \ref{SEC:GWP2}, 
we present a proof of  local well-posedness of \eqref{SNLW13} 
based on Sobolev's inequality
and construct a solution $v$ to \eqref{SNLW13}
in $C([0, T]; H^{1-\eps}(\T^2))$
for any $\eps >0$, where $T = T(\o)$ is an almost surely positive local existence time.
In this argument, we assume a priori  that a solution $v$ belongs
only to $C([0, T]; H^{1-\eps}(\T^2))$ (without intersecting with any auxiliary function space).
As a consequence, we obtain {\it unconditional uniqueness}
for the solution $v$ to \eqref{SNLW13}.
Unconditional uniqueness is a concept of uniqueness which does not depend on how solutions are constructed; see \cite{Kato}.
As a result, we obtain the uniqueness of the limiting process $u = \Phi + v$
in the entire class:
\[\Phi + C([0, T]; H^{1-\eps}(\T^2)).\]

\noi
Compare this with the solutions constructed in \cite{GKO1}, 
where we assume a priori that they also belong
 to some Strichartz space such that 
 the uniqueness statement in \cite{GKO1} is only conditional
 (namely in $C([0, T]; H^{1-\eps}(\T^2))$ intersected with the Strichartz space).

\smallskip

\noi
(ii) 
Let $(u, \dt u)$ the limiting process be constructed in Theorem \ref{THM:GWP2}.
Then, as a consequence of 
 Bourgain's invariant measure argument, 
 we obtain the following logarithmic growth bound:
\[ \| (u(t), \dt u(t)) \|_{\H^{-\eps}} \leq C(\o) \big(\log (1 + t)\big)^\frac{k}{2}\]

\noi
for any $t \geq 0$.
See \cite{ORTz} for details.

\end{remark}

\subsection{Remarks and comments}

(i) 
The stochastic nonlinear wave equations 
 have been studied extensively
in various settings; 
see \cite[Chapter 13]{DPZ14} for the references therein.
In recent years, we have witnessed
a rapid progress on the theoretical understanding 
of SNLW with singular stochastic forcing.
Since the work \cite{GKO1}
on local well-posedness of the renormalized SNLW on $\T^2$, 
there have been a number of works
on the subject:
 SNLW with a power-type nonlinearity on $\T^2$ and $\T^3$
 \cite{GKO2, ORTz, OOR, OO, OOT1, Bring, OOT2}
 and SNLW with trigonometric and exponential
 nonlinearities on $\T^2$
\cite{ORSW1, ORW, ORSW2}.
See also \cite{OT2, OPTz, OOTz}
for a related study on the deterministic NLW with random initial data.
We also  mention the work \cite{Deya1, Deya2} by Deya
on SNLW with 
more singular (both in space and time) noises on bounded domains in $\R^d$
and the work  \cite{Tolomeo2} by the fourth author on global well-posedness of the 
renormalized cubic SNLW on $\R^2$.

\smallskip

\noi
(ii)
In \cite{Tolomeo1}, the fourth author introduced 
a new approach to establish unique
ergodicity of Gibbs measures
for stochastic dispersive/hyperbolic equations.
In particular, ergodicity of the 
Gibbs measures was shown in \cite{Tolomeo1} for 
the cubic SdNLW on $\T$
and the cubic stochastic damped nonlinear beam equation on $\T^3$.
More recently, the fourth author 
further developed the methodology
and managed to prove 
ergodicity of the hyperbolic $\Phi^{k+1}_2$-model \eqref{SNLW11} 
for any odd integer $k \geq 3$; see
 \cite{Tolomeo3}.

\smallskip

\noi
(iii) For simplicity of the presentation, 
we only consider the regularization by  the sharp frequency cutoff $\P_N$ in this paper.
A straightforward modification allows us to treat
 regularization by a smooth mollifier.
Furthermore, by a standard argument, 
we can show that 
the limiting processes obtained through 
 regularization by a smooth mollifier
 agree with
 the limiting processes constructed in Theorems
\ref{THM:GWP1} and  \ref{THM:GWP2}
via the sharp frequency cutoff $\P_N$.
See \cite{OOTz} for such an argument
in the context of the deterministic NLW
with random initial data.

\section{Preliminary lemmas}
\label{SEC:2}

In this section, we introduce some notations and go over basic lemmas.

\subsection{Preliminary results from stochastic analysis}


In this subsection, 
by recalling some basic tools from probability theory and Euclidean quantum field theory
(\cite{Kuo, Nu, Shige, Simon}), 
we establish some preliminary estimates on the stochastic convolutions
and their Wick powers.
First, 
recall the Hermite polynomials $H_k(x; \s)$ 
defined through the generating function:
\begin{equation*}
F(t, x; \s) \stackrel{\text{def}}{=}  e^{tx - \frac{1}{2}\s t^2} = \sum_{k = 0}^\infty \frac{t^k}{k!} H_k(x;\s).
 \end{equation*}
	
\noi
For readers' convenience, we write out the first few Hermite polynomials:
\begin{align*}
\begin{split}
& H_0(x; \s) = 1, 
\quad 
H_1(x; \s) = x, 
\quad
H_2(x; \s) = x^2 - \s,   
\quad
 H_3(x; \s) = x^3 - 3\s x.
\end{split}
\end{align*}

Next, we recall the Wiener chaos estimate.
Let $(H, B, \mu)$ be an abstract Wiener space.
Namely, $\mu$ is a Gaussian measure on a separable Banach space $B$
with $H \subset B$ as its Cameron-Martin space.
Given  a complete orthonormal system $\{e_j \}_{ j \in \N} \subset B^*$ of $H^* = H$, 
we  define a polynomial chaos of order
$k$ to be an element of the form $\prod_{j = 1}^\infty H_{k_j}(\jb{x, e_j})$, 
where $x \in B$, $k_j \ne 0$ for only finitely many $j$'s, $k= \sum_{j = 1}^\infty k_j$, 
$H_{k_j}$ is the Hermite polynomial of degree $k_j$, 
and $\jb{\cdot, \cdot} = \vphantom{|}_B \jb{\cdot, \cdot}_{B^*}$ denotes the $B$-$B^*$ duality pairing.
We then 
denote the closure  of the span of
polynomial chaoses of order $k$ 
under $L^2(B, \mu)$ by $\mathcal{H}_k$.
The elements in $\H_k$ 
are called homogeneous Wiener chaoses of order $k$.
We also set
\[ \H_{\leq k} = \bigoplus_{j = 0}^k \H_j\]

\noi
 for $k \in \N$.

Let $L = \Dl -x \cdot \nabla$ be 
 the Ornstein-Uhlenbeck operator.\footnote{For simplicity, 
 we write the definition of the Ornstein-Uhlenbeck operator $L$
 when $B = \R^d$.}
Then, 
it is known that 
any element in $\mathcal H_k$ 
is an eigenfunction of $L$ with eigenvalue $-k$.
Then, as a consequence
of the  hypercontractivity of the Ornstein-Uhlenbeck
semigroup $U(t) = e^{tL}$ due to Nelson \cite{Nelson2}, 
we have the following Wiener chaos estimate
\cite[Theorem~I.22]{Simon}.
See also \cite[Proposition~2.4]{TTz}.

\begin{lemma}\label{LEM:hyp}
Let $k \in \N$.
Then, we have
\begin{equation*}
\|X \|_{L^p(\O)} \leq (p-1)^\frac{k}{2} \|X\|_{L^2(\O)}
 \end{equation*}
 
 \noi
 for any $p \geq 2$
 and any $X \in \H_{\leq k}$.

\end{lemma}

Before proceeding further, we recall the following corollary to 
the Garsia-Rodemich-Rumsey inequality
(\cite[Theorem A.1]{FV}).

\begin{lemma}\label{LEM:GRR}

Let $(E, d)$ be a metric space.
Given $u \in C([0, T]; E)$, 
suppose that there exist
$c_0 > 0$, $\ta \in(0, 1)$, and $\al > 0$
such that 
\begin{align}
\int_{t_1}^{t_2}
\int_{t_1}^{t_2} \exp \bigg\{c_0 \bigg( \frac{d(u(t), u(s))}{|t-s|^\ta}\bigg)^\al\bigg\} dt ds
= : F_{t_1, t_2} < \infty
\label{G1}
\end{align}

\noi
for any $0 \leq t_1 \leq t_2 \leq T$ with $t_2 - t_1 \leq 1$.
Then, we have
\begin{align}
 \exp \bigg\{\frac{c_0}{C} \bigg( 
 \sup_{t_1 \leq s < t \leq t_2} \frac{d(u(t), u(s))}{\zeta(t-s)}\bigg)^\al\bigg\} 
\leq  \max( F_{t_1, t_2}, e)
\label{G2}
\end{align}

\noi
for any $0 \leq t_1 \leq t_2 \leq T$
with $t_2 - t_1 \leq 1$,  
where $\zeta(t)$ is defined by 
\begin{align}
 \zeta(t) = \int_0^t \tau^{\ta - 1} \bigg\{\log\Big( 1+ \frac 4{\tau^2}\Big)  \bigg\}^\frac{1}{\al} d\tau. 
 \label{G2a}
\end{align}

\end{lemma}

When $\al = 2$, Lemma \ref{LEM:GRR} reduces to Corollary A.5 in \cite{FV}.
While Lemma \ref{LEM:GRR} for general $\al > 0$ follows 
in an analogous manner, 
 we present a proof 
for readers' convenience.

\begin{proof}
Let $\Psi(t) = e^{c_0 t^\al} - 1$ and $\textrm{p}(t) = t^\ta$.
Then,  from 
the Garsia-Rodemich-Rumsey inequality
(\cite[Theorem A.1]{FV}) with \eqref{G1}, 
we obtain
\begin{align}
d(u(t_1), u(t_2)) \leq 8\ta c_0^{-\frac{1}{\al}} \int_0^{t_2 - t_1}
t^{\ta - 1} \bigg\{ \log\Big( 1 + \frac{4 F_{t_1, t_2}}{t^2}\Big) \bigg\}^\frac{1}{\al} dt .
\label{G3}
\end{align}

\noi
Note that we have
\begin{align}
\log(1+AB) \leq \log(1+A) + \log B \leq 2 \log(1+A) \cdot \log B
\label{G4}
\end{align}

\noi
for $A \geq e - 1$ and $B \geq e$.
Then, it follows from \eqref{G3} and \eqref{G4} with \eqref{G2a} 
that 
\begin{align}
d(u(t_1), u(t_2)) \leq C \ta c_0^{-\frac{1}{\al}} 
\zeta (t_2 - t_1)
\big(\log  (\max( F_{t_1, t_2}, e))\big)^\frac{1}{\al}, 
\label{G5}
\end{align}

\noi
provided that  $\frac{4}{(t_2 - t_1)^2} \geq e-1$,
which is certainly satisfied for $0 < t_2 - t_1 \leq 1$.
The desired estimate \eqref{G2} 
follows directly from \eqref{G5}.
\end{proof}

Let $\Psi$ and $\Phi$ 
be the stochastic convolutions
defined in \eqref{PsiN} and \eqref{PhiN}, 
respectively.
Then, using standard stochastic analysis with the Wiener chaos estimate (Lemma \ref{LEM:hyp}), 
we have the following regularity and convergence result.

\begin{lemma}\label{LEM:Psi}
Let $Z = \Psi$ or  $\Phi$.
Given  $k \in \N$
and $N \in \N$, let $:\! Z_N^k \!:
\, = \, :\! (\P_N Z)^k \!:$
denote the truncated Wick power defined
in \eqref{Herm1} or \eqref{Herm3}, respectively.
Then, 
given any  $T,\eps>0$ and finite $p \geq 1$, 
 $\{ \, :\! Z_N^k \!: \, \}_{N\in \N}$ is a Cauchy sequence in $L^p(\O;C([0,T];W^{-\eps,\infty}(\T^2)))$,
 converging to some limit $:\!Z^k\!:$ in $L^p(\O;C([0,T];W^{-\eps,\infty}(\T^2)))$.
Moreover,  $:\! Z_N^k \!:$  converges almost surely to the same  limit in $C([0,T];W^{-\eps,\infty}(\T^2))$.
Given any finite $q\geq 1$, we   have 
the following tail estimate:
\begin{align}
P\Big( \|:\! Z^k \!:\|_{L^q_T W^{-\eps, \infty}_x} > \ld\Big) 
\leq C\exp\bigg(-c \frac{\ld^{\frac{2}{k}}}{T^{1 + \frac{2}{qk}}}\bigg)
\label{P0}
\end{align}

\noi
for any $T \geq 1$ and $\ld > 0$.
When $q = \infty$, we  also have 
the following tail estimate:
\begin{align}
P\Big( \|:\! Z^k \!:\|_{L^\infty ([j, j+1]; W^{-\eps, \infty}_x)}> \ld\Big) 
\leq C\exp\bigg(-c \frac{\ld^{\frac{2}{k}}}{j+1}\bigg)
\label{P0z}
\end{align}

\noi
for any $j \in \Z_{\ge 0}$ and $\ld > 0$.

\end{lemma}

\begin{proof}
In the following, we briefly discuss the case of  the stochastic convolution $\Psi$
associated with  the linear wave operator.
A straightforward modification yields the corresponding result for $\Phi$.
As for the convergence part of the statement, 
see \cite[Proposition~2.1]{GKO1} 
and  \cite[Lemma~3.1]{GKO2} for the details.
As for the exponential tail estimate \eqref{P0}, 
by repeating the argument
in the proof of \cite[Proposition~2.1]{GKO1}, 
we have
\begin{align}
\E\big[| \jb{\nb}^{-\eps} :\! \Psi^{k} (x, t) \!:  |^2 \big] 
& \les \sum_{n_1, \ldots, n_{k} \in \mathbb{Z}^2} \frac{t^k}{\langle n_1
   \rangle^2 \cdots \langle n_{k} \rangle^2 \langle n_1 + \cdots + n_{k}
   \rangle^{2\eps}} \le C_\eps t^k
\label{P0z1}
\end{align}

\noi
for any $\eps > 0$, uniformly in 
$x \in \T^2$ and $t \geq 0$.
Then, Minkowski's integral inequality and the Wiener chaos estimate (Lemma \ref{LEM:hyp}), 
we obtain
\begin{align}
 \Big\| \| :\! \Psi^{k}  \!:\|_{L^q_T W^{-\eps, \infty}_x}\Big\|_{L^p(\O)}
\les p ^\frac{k}{2} T^{\frac{k}{2} + \frac{1}{q}}
\label{P0a}
\end{align}

\noi
for any sufficiently large $p \gg1 $ (depending $q \geq 1$).
The exponential tail estimate \eqref{P0}
follows from \eqref{P0a} and Chebyshev's inequality
(see also Lemma 4.5 in \cite{TzBO}).

Fix $j \in \Z_{\ge 0}$ and $\ld > 0$. Then, we have
\begin{align}
\begin{split}
P\Big( \|:\! \Psi^k \!:\|_{L^\infty ([j, j+1]; W^{-\eps, \infty}_x)}
& > \ld\Big) 
\leq
 P\Big( \|:\! \Psi^k(j) \!:\|_{ W^{-\eps, \infty}_x}> \tfrac{\ld}{2}\Big) \\
+ 
& P\Big( \sup_{t \in [j, j+1]}\|:\! \Psi^k(t) \!: - :\! \Psi^k(j) \!:\|_{ W^{-\eps, \infty}_x}> \tfrac \ld2\Big). 
\end{split}
\label{P0b}
\end{align}

\noi
In view of \eqref{P0z1}, we see 
that 
the first term on the right-hand side of \eqref{P0b}
is controlled by the right-hand side of \eqref{P0z}.
As for the second term on the right-hand side
of \eqref{P0b}, 
we first recall from  the proof of \cite[Proposition~2.1]{GKO1}
that 
\begin{align*}
\Big\| |h|^{-\rho} \|\dl_h (:\! \Psi^{k} (t) \!:) \|_{W^{-\eps, \infty}_x} \Big\|_{L^p(\O)}
\les p^\frac{k}{2} (j+1)^\frac{k}{2}
\end{align*}

\noi
for any sufficiently large $p \gg1 $, $t \in [j, j+1]$, 
and $|h| \leq 1$, where $\dl_h f(t) = f(t+h) - f(t)$ and $0 < \rho < \eps$.
Then, by applying Lemma 4.5 in \cite{TzBO}, 
we obtain the following exponential bound:
\begin{align}
\E \Bigg[\exp \bigg\{(j+1)^{-1} \bigg( \frac{
\|:\! \Psi^{k} (\tau_2) - :\! \Psi^{k} (\tau_1) \!: \|_{W^{-\eps, \infty}_x}}{|\tau_2-\tau_1|^\rho}\bigg)^\frac{2}{k}\bigg\}\Bigg] 
\leq C < \infty, 
\label{P0d}
\end{align}

\noi
uniformly in  $j \leq \tau_1 < \tau_2 \leq  j+1$ (and $j \in \Z_{\ge 0}$).
By integrating~\eqref{P0d} in $\tau_1$ and $\tau_2$, 
this verifies
the hypothesis~\eqref{G1} of Lemma \ref{LEM:GRR}
(under an expectation).
Finally, applying 
Lemma~\ref{LEM:GRR} 
and then  Chebyshev's inequality, we conclude that 
\begin{align*}
P\Big( \sup_{t \in [j, j+1]}\|:\! \Psi^k(t) \!: - :\! \Psi^k(j) \!:\|_{ W^{-\eps, \infty}_x}> \tfrac \ld2\Big)
\leq C\exp\bigg(-c \frac{\ld^{\frac{2}{k}}}{j+1}\bigg).
\end{align*}

\noi
This proves \eqref{P0z}.
\end{proof}

In order to prove  Theorem \ref{THM:GWP1}, 
Lemma \ref{LEM:Psi} is not sufficient.
The following lemma shows
a  finer regularity property of $ \Psi$, 
namely, it is only logarithmically divergent from 
being a function. 
We recall that the $I$-operator depends
on the underlying $0 < s < 1$ and $N \in \N$.

\begin{lemma} \label{LEM:log}
Let $\Psi$ be as in \eqref{PsiN}
and fix $ 0 < s < 1$.
Then, given any $x \in \T^2$ and $t \in \R_+$, 
$I\Psi(x, t)$ is a mean-zero Gaussian random variable
with variance bounded by $C_0 t \log N$, 
where the constant $C_0$ 
is independent of  $x \in \T^2$ and $t \in \R_+$.

\end{lemma}

\begin{proof} 
Given any $x \in \T^2$ and $t \in \R_+$, 
$I\Psi(x, t)$ is obviously a mean-zero Gaussian random variable
(if the variance is finite).
By writing  $\Psi = \P_N \Psi + \P_N^\perp \Psi$, 
we separately estimate  the contributions from 
$\P_N \Psi$ and $ \P_N^\perp \Psi$.
For the low frequency part, we have $I \P_N \Psi = \P_N \Psi$
and thus from  \eqref{sig}, we have
\begin{align*}
\E \big[ (I \P_N  \Psi(x, t))^2\big] 
=\E \big[ (\P_N \Psi (x, t))^2 \big]
 \sim t \log N
\end{align*}

\noi
uniformly in $x \in \T^2$.
For the high frequency part, 
it follows from 
 \eqref{PsiN}, and \eqref{I0a} that 
  \begin{align*}
\E \big[(I \P_N^\perp \Psi (x, t))^2 \big]
& = 
\int_0^t \sum_{|n|> N} 
\E\big[ |\ft \Psi (n, t')|^2 \big] m_N^2(n)dt' \\
& \les t 
\sum_{|n|> N} \frac{N^{2-2s}}{|n|^{4-2s}}  \\
& \sim t, 
\end{align*}

\noi
uniformly in $x \in \T^2$.
This proves Lemma \ref{LEM:log}.
\end{proof}

\subsection{Product estimates}
We recall the following product estimates.
See \cite{GKO1} for the proof.

\begin{lemma}\label{LEM:bilin}
 Let $0\le s \le 1$.

\smallskip

\noi
\textup{(i)} Suppose that 
 $1<p_j,q_j,r < \infty$, $\frac1{p_j} + \frac1{q_j}= \frac1r$, $j = 1, 2$. 
 Then, we have  
\begin{equation*}  
\| \jb{\nb}^s (fg) \|_{L^r(\T^d)} 
\les \Big( \| f \|_{L^{p_1}(\T^d)} 
\| \jb{\nb}^s g \|_{L^{q_1}(\T^d)} + \| \jb{\nb}^s f \|_{L^{p_2}(\T^d)} 
\|  g \|_{L^{q_2}(\T^d)}\Big).
\end{equation*}

\smallskip

\noi
\textup{(ii)} 
Suppose that  
 $1<p,q,r < \infty$ satisfy the scaling condition:
$\frac1p+\frac1q\leq \frac1r + \frac{s}d $.
Then, we have
\begin{align*}
\big\| \jb{\nb}^{-s} (fg) \big\|_{L^r(\T^d)} \les \big\| \jb{\nb}^{-s} f \big\|_{L^p(\T^d) } 
\big\| \jb{\nb}^s g \big\|_{L^q(\T^d)}.  
\end{align*}

\end{lemma}

Note that
while  Lemma \ref{LEM:bilin} (ii) 
was shown only for 
$\frac1p+\frac1q= \frac1r + \frac{s}d $
in \cite{GKO1}, 
the general case
$\frac1p+\frac1q\leq \frac1r + \frac{s}d $
follows from the inclusion $L^{r_1}(\T^d)\subset L^{r_2}(\T^d)$
for $r_1 \geq r_2$.

\section{$I$-method 
for  the renormalized cubic SNLW}
\label{SEC:GWP1}

In this section, 
we prove  global well-posedness
of  the renormalized cubic SNLW \eqref{SNLW6} on~$\T^2$
(Theorem \ref{THM:GWP1}).
In Subsection~\ref{SUBSEC:3.1}, 
we go over preliminary estimates.
Then, we present a proof of  Theorem \ref{THM:GWP1} in Subsection \ref{SUBSEC:3.2}.

\subsection{Commutator and other preliminary estimates}
\label{SUBSEC:3.1}

In the following, we fix $N \in \N$ and $0 < s < 1$ and set\footnote{Recall that the $I$-operator
also depends on $0 < s < 1$.} $I = I_N$.
Moreover, we use the following notations:
\begin{align}
 f_{\les N} = \P_{\frac{N}{3}} f
\qquad \text{and}\qquad  f_{\ges N} = \P_{\frac{N}{3}}^\perp f = f - f_{\les N}.
\label{Q0}
\end{align}


We first go over basic commutator estimates in Lemmas \ref{LEM:C1}, \ref{LEM:C2}, 
and \ref{LEM:C3}.

\begin{lemma}\label{LEM:C1}
Let $\frac 23 \le s < 1$. Then, we have  
\begin{align}
\|(If)^k - I(f^k)\|_{L^2} \les N^{-1+k(1-s)} \|If\|^k_{H^1}
\label{Q1}
\end{align}

\noi
for $k = 1, 2, 3$.

\end{lemma}

\begin{proof}
By the definition of the $I$-operator and \eqref{Q0},  we have 
 $I(f_{\les N}^k) = f_{\les N}^k$
for $ k = 1, 2, 3$.
Thus, we have
\begin{align}
\begin{split}
(If)^k - I(f^k) 
& = \big(I(f_{\les N} + f_{\ges N})\big)^k - I\big((f_{\les N} + f_{\ges N})^k\big) \\
& = \big(f_{\les N} + I(f_{\ges N})\big)^k - I\big((f_{\les N} + f_{\ges N})^k\big)\\
& = \underbrace{f_{\les N}^k - I\big(f_{\les N}^k\big)}_{=0} + 
\sum_{j=0}^{k-1} {k \choose j}
\Big(f_{\les N}^j(If_{\ges N})^{k-j} - I\big(f_{\les N}^j f_{\ges N}^{k-j}\big)
\Big).
\end{split}
\label{Q2}
\end{align}

In the following, we use H\"older's inequality
with 
$\frac 12 = \frac jq + \frac{1}{2+\dl}$
for (i) 
 some large but finite $q\gg1 $ and small $\dl > 0$ when $j \geq 1$
 and (ii) $q = \infty$ and $\dl = 0$ when $j = 0$.
 Then, by 
H\"older's and  Sobolev's inequalities, we have
\begin{align}
\begin{split}
\|f_{\les N}^j (If_{\ges N})^{k-j}\|_{L^2}
& \le 
\|f_{\les N}\|_{L^{q}}^{j}
\|If_{\ges N}\|_{L^{(2+\dl) (k - j)}}^{k-j} \\
& \les 
 \|f_{\les N} \|_{H^{1}}^{j} 
\|If_{\ges N}\|_{H^{1 - \frac{2}{(2+\dl) (k - j)}}}^{k-j}\\
 & \les N^{-1+\eps} \|If\|_{H^1}^k
\end{split}
\label{Q3}
\end{align}

\noi
for some small $\eps > 0$.
Proceeding similarly
with  the boundedness of the multiplier $m_N$
and~\eqref{I1}, we have  
\begin{align}
\begin{split}
\big\|I\big(f_{\les N}^j f_{\ges N}^{k-j}\big)\big\|_{L^2} 
& \les \|f_{\les N}^j f_{\ges N}^{k-j}\|_{L^2} \\
& \le \|f_{\les N}\|_{L^q}^j \|f_{\ges N}\|_{L^{(2+\dl)(k-j)}} \\
& \les 
\|I f_{\les N}\|_{H^1}^j
\| f_{\ges N} \|_{H^{1 - \frac{2}{(2+\dl) (k - j)}}}^{k-j}\\
& \les 
N^{-1 + k(1-s)} 
\|I f_{\les N}\|_{H^1}^j
\| f_{\ges N} \|_{H^s}^{k-j}\\
& \les N^{-1 + k(1-s)} \|If \|_{H^1}^k
\end{split}
\label{Q4}
\end{align}

\noi
since $ \frac 23 \le s < 1$.
Therefore, the desired estimate \eqref{Q1}
follows from \eqref{Q2}, \eqref{Q3}, and~\eqref{Q4}.
\end{proof}

\begin{lemma} \label{LEM:C2}
 Let $0< \s < 1$. 
 Given $\dl > 0$, there exist small $\s_0 = \s_0(\dl) > 0$
 and large $p = p(\dl) \gg1$ such that  
\begin{align}
\norm{(If)(Ig) - I(fg)}_{L^2} 
\les N^{ - \frac{1-\s}{2} + \dl } \|f\|_{H^{1-\s}}\|g\|_{W^{-\s_0, p}}
\label{Q5a}
\end{align}

\noi
for any sufficiently large $N \gg1$.
\end{lemma}

\begin{proof}
By writing $f = f_{\les N^{\frac 12}} + f_{\ges N^{\frac12}}$ 
 and $g= g_{\les N} + g_{\ges N}$, we have 
 \begin{align}
 \begin{split}
 (If)(Ig) - I(fg) 
 &= \Big\{ (If_{\les N^{\frac12}}) (Ig_{\les N}) 
 - I(f_{\les N^{\frac12}}g_{\les N})\Big\}\\
& \hphantom{X}
+ \Big\{ (If_{\les N^{\frac12}})(Ig_{\ges N}) - 
I(f_{\les N^{\frac12}}g_{\ges N})\Big\}\\
& \hphantom{X}
+ (If_{\ges N^{\frac12}})(Ig)\\
& \hphantom{X}
- I(f_{\ges N^{\frac12}}g)\\
& =: B_1 + B_2 + B_3 + B_4.
\end{split}
\label{Q5}
 \end{align}

From the definition of the $I$-operator with \eqref{Q0}, 
we see that 
\begin{align}
B_1 = 0
\label{Q6}
\end{align}

\noi
for any sufficiently large $N \gg1 $
since 
$\supp\big\{ \F(f_{\les N^{\frac12}}g_{\les N})\big\}
\subset \{ n \in \Z^2 : |n| \leq  \frac 56N\}$
for $N \gg 1$.

For $|n_1| \les N^\frac{1}{2}$ and $|n_2| \ges N$, 
from the mean value theorem with \eqref{I0a}, we have
\begin{align}
|m(n_1+n_2)-m(n_2)| \les  N^{1-s}|n_2|^{-2+s}
|n_1|.
\label{Q7a}
\end{align}

\noi
Let $\star _n= \{ n_1, n_2 \in \Z^2: n = n_1 + n_2, \, |n_1| \leq \frac{N^\frac{1}{2}}3,
\, |n_2| > \frac N3 \}$.
By  \eqref{Q7a}, the fact that  $m(n_1) \equiv 1$
on $\star_n$, 
and 
 Young's inequality
followed by Cauchy-Schwarz inequality (in $n_1$),
we have
 \begin{align}
\begin{split}
\| B_2\|_{L^2} 
& = \bigg\|\sum_{\star_n }
\big(m(n_2)-m(n_1+n_2)\big)\ft f (n_1)\ft g(n_2)\bigg\|_{\l^2_n}\\
& \les N^{1-s}
\Bigg\| \sum_{\star_n} \frac{\jb{n_1}^{1+\s+\delta}}{|n_2|^{2-s-\delta}}
\bigg|\frac{\ft f(n_1)}{\jb{n_1}^{\s+\delta}}\bigg| \bigg|\frac{\ft g(n_2)|}{|n_2|^{\delta}}\bigg|
\Bigg\|_{\l^2_n}\\
& \les N^{-\frac {1-\s}2  + \frac 32 \dl}
\| \jb{n_1}^{-\s - \dl} \ft f(n_1) \|_{\l^1_{n_1}}
\|g\|_{H^{-\dl} }\\
& \les N^{-\frac {1-\s}2  + \frac 32 \dl}
\|f\|_{H^{1-\s}}\|g\|_{H^{-\dl} }.
\end{split}
\label{Q7b}
\end{align}

As for $B_3$, by H\"older's inequality, Sobolev's embedding theorem, and 
applying \eqref{I2} twice, 
we have
\begin{align}
\begin{split}
\|B_3\|_{L^2} 
& \le \|If_{\ges N^{\frac12}}\|_{L^2} \|Ig\|_{L^\infty} \\
& \les N^{-\frac{1-\s}{2}} \|f \|_{H^{1-\s}}\|Ig\|_{W^{3\delta,\delta^{-1}}}\\
& \les   N^{-\frac{1-\s}{2}+ 4 \delta} \|f\|_{H^{1-\s}}\|g\|_{W^{-\delta,\delta^{-1}}}
\end{split}
\label{Q8}
\end{align}

\noi
for $\dl > 0$ sufficiently small.

Lastly, from 
\eqref{I2}
and Lemma \ref{LEM:bilin}\,(ii), we have
\begin{align}
\begin{split}
\|B_4\|_{L^2} 
& \les  N^{2\dl} \| f_{\ges N^\frac12} g \|_{H^{-2\dl}}\\
& \les N^{2\dl}  \| f_{\ges N^\frac{1}{2}}\|_{H^{2\dl}} \|g\|_{W^{-2\dl , \dl^{-1}}}\\
& \les N^{-\frac {1-\s}2  + 3\dl}  \| f_{\ges N^\frac{1}{2}}\|_{H^{1-\s}} \|g\|_{W^{-2\dl , \dl^{-1}}}
\end{split}
\label{Q9}
\end{align}

\noi
for $\dl > 0$ sufficiently small.

Putting \eqref{Q5}, \eqref{Q6}, \eqref{Q7b}, 
and \eqref{Q8}, and \eqref{Q9} together, 
we obtain \eqref{Q5a}.
 \end{proof}

 From Lemmas \ref{LEM:C1} and \ref{LEM:C2}, 
 we obtain the following commutator estimate.
 For our application, 
 we will use this lemma with $g =\, \,   :\!\Psi^{3-k}\!:\,$.
 
\begin{lemma}\label{LEM:C3}
Let $\frac 23 \leq s < 1$
and  $k =  1,  2$. 
 Given $\dl > 0$, there exist small $\s_0 = \s_0(\dl) > 0$
 and $p = p(\dl) \gg1$ such that  
\begin{align}
\| I(f^k g) - (If)^k I g \|_{L^2} 
\les
N^{-\frac{1-k(1-s)}{2} + \dl } \|If\|_{H^1}^k \|g\|_{W^{-\s_0 ,p}}
\label{Q31}
\end{align}

\noi
for any sufficiently large $N \gg1$.

\end{lemma}

\begin{proof}
By the triangle inequality, we have 
\begin{align}
\begin{split}
\| I ( f^kg ) - (If)^k Ig\|_{L^2} 
&\le \| I(f^kg) -I(f^k)Ig\|_{L^2}
+ \big\|\big(I(f^k) - (If)^k\big)Ig\big\|_{L^2}\\
& =: D_1 + D_2.
\end{split}
\label{Q32}
\end{align}

By Sobolev's inequality (with $s > \frac 12$)
and the fractional Leibniz rule (Lemma \ref{LEM:bilin}\,(i)), 
we have
\begin{align} 
\|f^{k}\|_{H^{1-k(1-s)}} 
\les \|f^k\|_{W^{s,\frac{2}{1+(k-1)(1-s)}}}
\les \|f\|_{H^{s}} \|f\|_{L^{\frac2{1-s}}}^{k-1} \les\|f\|_{H^{s}}^k.
\label{Q33a}
\end{align}

\noi
Thus, by  Lemma \ref{LEM:C2} with $\s = k(1-s)$, 
\eqref{Q33a}, 
and \eqref{I1}, given $\dl > 0$, we have
\begin{align}
\begin{split}
 \|D_1\|_{L^2} 
 & \les N^{-\frac{1-k(1-s)}2 + \dl} 
\|f\|^k_{H^{s}} \|g\|_{W^{-\s_0, p }}\\
&  \les N^{-\frac{1-k(1-s)}2 + \dl} 
\|If\|^k_{H^{1}} \|g\|_{W^{-\s_0, p }}
\end{split}
\label{Q33}
\end{align}

\noi
for some small $\s_0 = \s_0 (\dl)>0$ and large $p = p (\dl) \gg1 $.
On the other hand, 
by H\"older's inequality, Lemma \ref{LEM:C1}, Sobolev's embedding theorem, and
\eqref{I2}, we have
\begin{align}
\begin{split}
\|D_2\|_{L^2} & 
\le \|I(f^k) - (If)^k \|_{L^2} 
\|Ig\|_{L^\infty} \\
& \les N^{-1+k(1-s)} \|If\|_{H^1}^k \|Ig\|_{W^{3\delta,\delta^{-1}}} \\
& \les N^{-1+k(1-s) + 4\delta} \|If\|_{H^1}^k 
\|g\|_{W^{-\delta,\delta^{-1}}}.
\end{split}
\label{Q34}
\end{align}

Putting \eqref{Q32}, \eqref{Q33}, and \eqref{Q34} together, 
we obtain \eqref{Q31}.
\end{proof}

We conclude this subsection by 
presenting useful
estimates for controlling the Gronwall part of our hybrid $I$-method argument.

\begin{lemma}\label{LEM:C4}
\textup{(i)}
 Let $k = 0,1$. Then, for any $0 \le \g  \le 1-s$, 
 we have 
 \begin{align*}
\bigg|\int_{\T^2} (\dt I v(t) )(I v(t))^k I w (t) \, dx \bigg|
\les N^\g\big(1 + E^{\frac34}(I \vec v )(t) \big) \| w(t) \|_{ W_x^{-\g, 4} }
\end{align*}

\noi
for any $t \geq 0$, 
where $E$ is the energy defined  in \eqref{Hamil}.

\smallskip

\noi
\textup{(ii)}
There exists $c>0$ such that 
\begin{align}
\begin{split}
\bigg|\int_{t_1}^{t_2} \int_{\T^2} (\dt  & Iv)   (Iv)^2 I w \, dx dt \bigg| \\
& \les 
\Bigg\{\int_{t_1}^{t_2} \bigg(E^{1+c\eta}(I\vec v)(t) + \frac{\eta}{(t-t_1)^\frac12}\bigg) dt\Bigg\}
\|I w\|_{L^{\eta^{-1}}_{[t_1, t_2],x}} , 
\end{split}
\label{Q42}
\end{align}

\noi
uniformly 
in $0 < \eta < \frac 18$
and $t_2 \ge t_1 \ge 0$, 
where  $L^p_{I, x} = L^p(I; L^p(\T^2))$
for a given time interval $I \subset \R_+$.

\end{lemma}

 For our application, 
 we will use  Part (i) 
with $w = \, :\!\Psi^{3-k}\!:\, $, $k = 0,1$, 
and Part (ii) with $w = \Psi$.

\begin{proof}(i) 
Let $k = 0, 1$. 
Then, by H\"older's inequality, \eqref{Hamil}, and \eqref{I2},  we have
\begin{align*}
\bigg|\int_{\T^2} (\dt  Iv(t)) (Iv(t))^k  Iw(t) \, dx\bigg| 
& \le\|\dt Iv(t)\|_{L^2} \|Iv(t) \|_{L^4}^k \|I w(t)\|_{L^{\frac 4 {2-k}}} \\
& \les\|\dt Iv(t)\|_{L^2} \|Iv (t)\|_{L^4}^k \|I w(t)\|_{L^4} \\
& \les  N^\g  E(I\vec v)^{\frac12 + \frac k4} (t)
\|w(t)\|_{W^{-\g , 4}}.
 \end{align*}

\smallskip
\noi
(ii) By interpolation with \eqref{Hamil}, we have 
 \begin{align*}
\|I v\|_{W^{\theta, \frac 4 {1+\theta}}}
\les \|Iv\|_{H^1}^\ta \|Iv\|_{L^4}^{1-\ta}
  \les E^\frac \ta2(I \vec v)  E^\frac {1-\ta}4(I\vec v)
= E^{\frac{1+\ta}{4}} (I\vec v)
 \end{align*}
 
 \noi
 for $0 \leq \ta \leq 1$.
 Then, by Sobolev's inequality, we have 
 \begin{align}
\|I v\|_{L^{\frac 4 {1- \theta}}}
  \les  E^{\frac{1+\ta}{4}} (I\vec v), 
\label{Q43}
\end{align}

\noi
where  the implicit constant is uniform in $\theta$ as long as $0\le \theta \le \theta_{\max} <1$. 
Set  $\theta = 4 \eta$.
 Then, by H\"older's inequality (in $x$), 
 \eqref{Hamil},  \eqref{Q43}, 
 and H\"older's inequality (in $t$), 
 we obtain
 \begin{align}
\begin{split}
\bigg|\int_{t_1}^{t_2} \int_{\T^2} (\dt Iv) (Iv)^2 I w \, dx dt\bigg| 
&\le \int_{t_1}^{t_2}
\|\dt Iv\|_{L^2_x} 
\|Iv\|_{L^4_x} \|Iv\|_{L^{\frac 4 {1-4\eta}}_x} \|Iw \|_{L_x^{\eta^{-1}}} dt \\
& \les \int_{t_1}^{t_2} E^{1 + \eta} (I \vec v) \|I w \|_{L_x^{\eta^{-1}}} dt \\
& \les \bigg(\int_{t_1}^{t_2} E^{\frac{1 + \eta}{1-\eta}}(I\vec v)dt\bigg)^{1-\eta}
\|I w \|_{L_{[t_1, t_2], x}^{\eta^{-1}}}, 
\end{split}
\label{Q44}
 \end{align}
 
 \noi
uniformly in  $0 < \eta < \frac 18$.

Next, we estimate the first factor on the right-hand side in \eqref{Q44}.
Let 
\[ p= p(\eta) = \frac{1-\eta}{1-2 \eta} 
\qquad 
\text{and}\qquad 
q = q(\eta) = \frac{1}{1-2\eta}.\]

\noi
This implies that 
\[ p' = \frac {1-\eta}{\eta}
\qquad 
\text{and}\qquad 
 q' = \frac{1}{2\eta}.\]

\noi
Then, 
by  H\"older's and Young's inequalities, we have
\begin{align*}
\bigg(\int_{t_1}^{t_2} f (t) dt \bigg)^{1-\eta} 
&\le \bigg(\int_{t_1}^{t_2} |f(t) |^p dt \bigg)^{\frac{1-\eta}{p}} (t_2 - t_1)^{\frac{1-\eta}{p'}}\\
&\le \frac 1{q}\bigg(\int_{t_1}^{t_2} |f(t)|^pdt \bigg)^{q\cdot \frac{1-\eta}{p}} + \frac{1}{q'} (t_2 - t_1)^{q'\cdot \frac{1-\eta}{p'}} \\
&= (1-2\eta) \int_{t_1}^{t_2} |f(t) |^{\frac{1-\eta}{1-2 \eta}} \, dt + 2\eta (t_2 - t_1)^\frac12.
\end{align*}

\noi
Applying this to 
 the first factor on the right-hand side in \eqref{Q44}, we obtain
\begin{align}
\bigg(\int_{t_1}^{t_2} E^{\frac{1+\eta}{1-\eta}}(I\vec v)(t)dt \bigg)^{1-\eta} 
\lesssim 
\int_{t_1}^{t_2} \bigg(E^{\frac{1+\eta}{1-2\eta}}(I\vec v)(t) + \frac{\eta}{(t-t_1)^\frac12}\bigg) dt.
\label{Q45}
\end{align}

\noi
Putting \eqref{Q44} and \eqref{Q45} together, 
we obtain \eqref{Q42}.
\end{proof}

\subsection{Proof of Theorem \ref{THM:GWP1}}
\label{SUBSEC:3.2}

In this subsection, we use the estimates 
in the previous subsection
and implement an iterative argument to 
construct a solution to \eqref{SNLW6} on a time interval $[0, T]$
for any given $T \gg 1$.
Unlike the usual application of the $I$-method
(where the parameter $N$ depends only on the target time $T\gg1 $), 
we will need to construct an increasing sequence $\{N_k \}_{k \in \Z_{\ge 0}}$ of parameters
over local-in-time intervals, which allows
us to proceed over a time interval of fixed length at each iteration step.

Fix $\frac 45 < s < 1$ and a target time $T\gg 1$.
Our main goal is to control
growth of  the modified energy $E(I\vec v) (t)$
on the time interval $[0, T]$.
We use the following short-hand notation
for the modified energy:
\[ E(t) = E(I\vec v )(t).\]

\noi
Then,  from \eqref{E1}, Lemmas \ref{LEM:C1}, \ref{LEM:C3}, 
and \ref{LEM:C4}, 
we have 
\begin{align} 
\begin{split}
E(t_2) & - E(t_1) \\
\les
& \int_{t_1}^{t_2} N^{-1+3(1-s)} E^2(t) dt \\
& +  \sum_{k=1}^2 \int_{t_1}^{t_2} N^{-\frac{1-k(1-s)}2 + \dl}
E^\frac {k+1} 2(t)\, \| \!:\!\Psi^{3-k}(t) \!:\! \|_{W^{-\s_0, p }_x} dt \\
& + \sum_{k=0}^1 \int_{t_1}^{t_2} N^\gamma (1 + E^{\frac34}(t) ) 
\|\!:\!\Psi^{3-k}(t)\!:\!\|_{ W_x^{-\gamma, 4}}  dt \\
& + 
\Bigg\{\int_{t_1}^{t_2} \bigg(E^{1+c\eta}(t) + \frac{\eta}{(t-t_1)^\frac12}\bigg) dt\Bigg\}
\|I \Psi\|_{L^{\eta^{-1}}_{[t_1, t_2],x}}
\end{split}
\label{Y1}
\end{align}

\noi
for any $t_2 \geq t_1 \geq 0$.

Before proceeding to the following  crucial proposition, let us introduce some notations.
Given $j \in \Z_{\ge 0}$, set $V_j = V_j(\o)$ by
\begin{align*}
V_j =  \max_{k = 1, 2} \|:\!\Psi^{3-k}\!:\|_{L^\infty_{[j, j+1]}W^{-\s_0, p}_x} 
+ \max_{k = 0, 1}\|:\!\Psi^{3-k}\!:\|_{L^\infty_{[j, j+1]}W^{-\g, 4}_x} 
\end{align*}

\noi
and define $V = V(\o)$ by 
\begin{align}
e^{V^\frac{1}{3}} = \sum_{j = 0}^\infty e^{-\ta j} e^{V_j^\frac{1}{3}}
\label{K2}
\end{align}

\noi
for some  $\ta > 0$.
Note that $V$ is almost surely finite, 
since,  
by applying 
\eqref{P0z} in 
Lemma \ref{LEM:Psi}
and choosing  $\ta  $ sufficiently large, we have 
\begin{align*}
 \E\Big[ 
e^{V^\frac{1}{3}} \Big] = \sum_{j = 0}^\infty e^{-\ta j} \E\Big[e^{V_j^\frac{1}{3}}\Big]
\leq \sum_{j = 0}^\infty e^{-\ta j} e^{c (j+1)} < \infty.
\end{align*}

Now, given $T\gg1$, set 
\begin{align}
M_T =  \max_{k = 1, 2} \|:\!\Psi^{3-k}\!:\|_{L^\infty_TW^{-\s_0, p}_x} 
+ \max_{k = 0, 1}\|:\!\Psi^{3-k}\!:\|_{L^\infty_TW^{-\g, 4}_x} .
\label{K2a}
\end{align}

\noi
Then, noting from \eqref{K2} that 
\[ V_j^\frac{1}{3} \leq V^\frac{1}{3} + \ta j, \] 

\noi
we have
\begin{align}
M_T =  \max_{j \leq T} V_j \le
 V^\frac{1}{3} + \ta T 
 \les V + T^3.
\label{K3}
\end{align}

We also define $R = R(\o)$ by 
\begin{align}
R = \sum_{N = 1}^\infty \sum_{j = 1}^\infty
e^{-\ta j \log N}
\int_0^j \int_{\T^2}e^{|I_N \Psi (x, t) |} dx dt, 
\label{K4}
\end{align}

\noi
where $I = I_N$ is the $I$-operator defined in \eqref{I0}.
Then, by applying Lemma \ref{LEM:log}
and choosing  $\ta  $ sufficiently large, we have 
\begin{align*}
\E[ R ] 
& = \sum_{N = 1}^\infty \sum_{j = 1}^\infty
e^{-\ta j \log N} 
\int_0^j \int_{\T^2}\E \Big[e^{|I_N \Psi (x, t) |}\Big] dx dt\\
& \les  \sum_{N = 1}^\infty \sum_{j = 1}^\infty
e^{-\ta j \log N}
j  e^{c j\log N} < \infty.
\end{align*}

\noi
Thus, $R$ is finite almost surely.
In the following, we assume that $R = R(\o) \geq 1$.

In the following, we fix $\o \in \O$
such that $V = V(\o) < \infty$ and $R = R(\omega) < \infty$
and prove global well-posedness by pathwise analysis.
The following proposition plays a fundamental role
in  our iterative argument to prove Theorem \ref{THM:GWP1}.

\begin{proposition} \label{PROP:LGWP}
Let $\frac 23 <  s < 1$,  $T \ge T_0 \gg 1$, and $N \in \N$.
Moreover, let $V = V(\o) < \infty$ and $R = R(\o) < \infty$
be as in \eqref{K2} and \eqref{K3}.
Then, there exist $\al = \al(s) $,  $\be = \be(s)> 0$
with $\al > \be$ such that 
if \begin{align}
E(t_0) \le N^\beta
\label{Y2}
\end{align}

\noi
for some $0 \le t_0 < T$, then
there exists
$\tau = \tau (s, N, T, V, R) = \tau (s, N, T, \o) >0 $
with  $\tau \le t_\ast(R) \le 1$
such that 
\begin{align}
E(t) \le N^\alpha
\label{Y3}
\end{align} 

\noi
for any $t $
such that  
$t_0 \le t \le \min(T, t_0 + \tau)$. 
\end{proposition}

\begin{proof}
Without loss of generality, we assume that $E(t) \geq 1$.  (This can be guaranteed by replacing
$E(t)$ by $E(t) + 1$.)
Then, from \eqref{Y1} with \eqref{K2a}, we have
\begin{align} 
\begin{split}
E(t) \,  - & \,  E(t_0) \\
\les
& \, \int_{t_0}^{t} N^{-1+3(1-s)} E^2(t') dt' \\
& +  M_T \sum_{k=1}^2 \int_{t_0}^{t} N^{-\frac{1-k(1-s)}2 + \dl}
E^\frac {k+1} 2(t') dt' \\
& + M_T  \int_{t_0}^{t} N^\gamma  E^{\frac34}(t') 
  dt' \\
& + 
\Bigg\{\int_{t_0}^{t} \bigg(E^{1+c\eta}(t') + \frac{\eta}{(t'-t_1)^\frac12}\bigg) dt'\Bigg\}
\|I \Psi\|_{L^{\eta^{-1}}_{[t_0, t],x}}
\end{split}
\label{K5}
\end{align}

\noi
for any $t \geq t_0$.

In the following, we  assume
\begin{align}
\max_{t_0 \leq \tau \leq t} E(\tau) \leq 100 N^{\al}
\label{K5a}
\end{align}

\noi
for some $t \geq t_0$, 
where $\al > \be$ is to be determined later.
Then, we show that \eqref{Y3} holds for this $t$.
%
It follows from  the continuity in time of $E(t)$
and~\eqref{Y2} with $\al > \be$
that there exists $t_1 > t_0$ sufficiently close to $t_0$
such that 
\eqref{K5a} holds true for $t_0 \le t  \le t_1$.

Letting $\eta^{-1} = n  \in \N$, 
it follows from 
\eqref{K4}
that
\begin{align*}
\| I_N \Psi\|_{L^n_{[t_1, t_2], x}}^n 
& = \int_{t_1}^{t_2}\int_{\T^2}
| I_N \Psi(x, t)|^n 
 dx dt 
 \leq n!  \int_0^T 
\int_{\T^2}
e^{| I_N \Psi(x, t)|}
 dx dt \\
& \le
n! e^{\ta T \log N} R.
\end{align*}

\noi
With $n! \leq n^n$, this implies
\begin{align*}
\| I_N \Psi\|_{L^n_{[t_1, t_2], x}}
& \le
n e^{\frac 1n \ta T \log N} R^\frac{1}{n}.
\end{align*}

\noi
We now choose 
\begin{align*}
n \sim \ta T \log N +  c \al \log (100 N) \sim T \log N \gg 1.
\end{align*}

\noi
Then, 
under the assumption   \eqref{K5a} and $\eta = n^{-1}$,
we can estimate the last term on the right-hand side of \eqref{K5}
as
\begin{align} 
\begin{split}
& \Bigg\{\int_{t_0}^{t} \bigg(E^{1+c\eta}(t') + \frac{\eta}{(t'-t_1)^\frac12}\bigg) dt'\Bigg\}
\|I \Psi\|_{L^{\eta^{-1}}_{[t_0, t],x}}\\
& 
\leq 
\int_{t_0}^{t} \bigg(E(t') n e^{\frac 1 n (\ta T\log N + c\al \log (100N )) } + \frac{e^{\frac 1n \ta T \log N} R^\frac{1}{n}}{(t'-t_1)^\frac12}\bigg) dt'\\
& 
\leq 
\int_{t_0}^{t} \bigg(T E(t') \log N  + \frac{R}{(t'-t_1)^\frac12}\bigg) dt', 
\end{split}
\label{K8}
\end{align}

\noi
where we used the assumption that $n \geq 1$ and $R = R(\o) \geq 1$
in the last step.

Next, we define $F$ by 
\begin{align}
 F(t) = \max_{t_0 \le \tau \le t}  E(\tau) - E(t_0) + \max(E(t_0), N^\be).
 \label{K8a}
\end{align}

\noi
Then, under the assumption \eqref{K5a}, we have
we have 
\begin{align}
N^\be \leq F(t) \leq 200 N^\al
\label{K9}
\end{align}

\noi
for $t_0 \le t \le t_1$.
In particular, we have
\begin{align}
\log F(t) \sim \log N.
\label{K10}
\end{align}

\noi
Moreover, from \eqref{K9}, we have
 \begin{align} 
\begin{cases}
  N^{-1+3(1-s)} F^2(t) \les N^{-\al}F^2(t) \leq F(t), \\
  N^{-\frac{1-2(1-s)}2 + \dl}F^\frac{3}{2}(t)
 \les N^{-\frac \al2 }F^\frac 32(t) \leq F(t),\\
  N^{-\frac{1-(1-s)}2 + \dl} F(t) \leq  F(t), \\
  N^\g  F^{\frac34}(t) \les 
  N^\g F^{-\frac{1}{4}}(t) F(t)
 \le F(t), 
\end{cases}
\label{K11}
\end{align}

\noi
provided that 
\begin{align}
\alpha \le 1 - 3(1-s) = -2 + 3s, \quad 
\dl \le \min\big(\tfrac{2s - 1 - \al}{2}, \tfrac s2\big),
\quad \text{and}\quad  \g \le \tfrac\be 4.
\label{K12}
\end{align}

\noi 
Here, $\g = \g(s) > 0$ is a small constant,  appearing in Lemma \ref{LEM:C4}.
The first condition in~\eqref{K12} with $\al > 0$ requires $s> \frac 23$.
Hence, from \eqref{K5} with \eqref{K8}, \eqref{K8a}, 
\eqref{K10}, and \eqref{K11}
followed by \eqref{K3}, we obtain 
\begin{align} 
\begin{split}
F(t) \,  & -  \,   F(t_0) \\
& \les
 (1 + M_T) \int_{t_0}^{t} F(t') dt' 
 + 
\int_{t_0}^{t} \bigg(T F(t') \log F(t')  + \frac{R}{(t'-t_1)^\frac12}\bigg) dt'\\
& \les
 (1 + V + T^3) \int_{t_0}^{t} F(t') dt' 
 + 
\int_{t_0}^{t} \bigg(T F(t') \log F(t')  + \frac{R}{(t'-t_1)^\frac12}\bigg) dt'\\
& \les
 (1 + V + R  + T)
\int_{t_0}^{t} \bigg( F(t') (\log F(t') + T^2)   + \frac{R}{(t'-t_1)^\frac12}\bigg) dt'
\end{split}
\label{K13}
\end{align}

\noi
for any $t_0 \le t \le t_1$ such that \eqref{K9} holds.
Denoting by $C_0$ the implicit constant in \eqref{K13}, 
we define $G$ by 
\begin{align}
 G(t) = F(t) - 2 C_0 R (t - t_0)^\frac{1}{2}.
 \label{K14}
\end{align}

\noi
Then, it follows from \eqref{K13} that 
\begin{align} 
\begin{split}
G(t)  -  G(t_0) 
\les (1 + V + R  + T)
\int_{t_0}^{t}  G(t') (\log G(t') + T^2)  dt'
\end{split}
\label{K15}
\end{align}

\noi
for any $t_0 \le t \le \min( t_1, t_0 + t_*(C_0, R))$
such that 
\begin{align}
2 C_0 R (t - t_0)^\frac{1}{2} \sim 1
\label{K15a}
\end{align}

\noi
(which guarantees $G(t) \sim F(t)$ in view of \eqref{K14}).

Now, note that the equation 
\[ \dt H(t) = \kk H(t) (\log H(t) + T^2)\]

\noi
has an explicit solution 
\[ H(t) = \exp\Big( e^{\kk t } (\log H(0) + T^2) - T^2\Big).\] 

\noi
Then, by comparison, we deduce from \eqref{K15} that 
\begin{align}
G(t) \leq  \exp \Big( e^{C(1 + V + R + T)  (t - t_0) } (\log G(t_0) + T^2) - T^2\Big).
\label{K16}
\end{align}

Recall from 
\eqref{K14} and 
\eqref{K8a} that 
$G(t_0) = N^\be$.
Then, under the condition
\begin{align}
e^{C(1 + V + R + T)  (t - t_0) } (\be \log N  + T^2) \le \al \log N +  T^2 - \log 2, 
\label{K17}
\end{align}

\noi
the bound \eqref{K16} implies 
\begin{align}
G(t) \le \tfrac 12 N^\al
\label{K18}
\end{align}

\noi
for any $t_0 \le t \le \min( t_1, t_0 + t_*(C_0, R))$.
Then, we conclude from 
\eqref{K8a}, 
\eqref{K14}, and  \eqref{K15a}
that 
\begin{align}
E(t) \le F(t) \le   N^\al
\label{K19}
\end{align}

\noi
for any $t_0 \le t \le \min( t_1, t_0 + t_*(C_0, R))$.
This in turn guarantees the conditions \eqref{K5a}
and~\eqref{K9}.
Therefore, 
by a standard continuity argument, 
we conclude that 
the bounds
\eqref{K18} and 
\eqref{K19} hold
for any $t_0 \le t \le t_0 + t_*(C_0, R)$
sufficiently close to $t_0$ such that 
the condition \eqref{K17} holds.

Finally, let us rewrite the condition \eqref{K17}.
Let $\al = \al(s) > \be = \be(s) $ satisfy the conditions \eqref{K12}.
Then, there exists small 
 $0 < \tau \leq  t_*(C_0, R)$
such that 
\begin{align}
 \al  - e^{C(1 + V + R + T)  \tau }\be \geq c_0 > 0.
 \label{K20}
\end{align}

\noi
Then, 
by choosing $\tau = \tau (s, N, T, V, R) >0 $ sufficiently small
such that 
\begin{align}
  e^{C(1 + V + R + T)  \tau } - 1
\leq \frac{c_0 \log N - \log 2}{T^2}, 
 \label{K21}
\end{align}

\noi
we can guarantee the condition \eqref{K17}
and hence the desired bound \eqref{K19}
for $0 \le t - t_0 \le \tau$.
This conclude the proof of Proposition \ref{PROP:LGWP}.
\end{proof}

\begin{remark}\label{REM:tau}\rm
By choosing $\tau \sim_{V, R} T^{-1}$ sufficiently small, 
we can  guarantee the condition~\eqref{K20}.
\end{remark}

We now present a proof of Theorem \ref{THM:GWP1}.
Fix  $\frac 45 < s <1$ and $T\gg1 $.
Moreover, we fix $\o \in \O$ such that $V = V(\o) < \infty$
and $R = R(\o)< \infty$.
Then, 
 let the parameters
$\al, \be,   \tau$ be as in Proposition~\ref{PROP:LGWP}.

Fix $N_0 \gg 1$ (to be determined later).
Then, for $k \in \Z_{\geq 0}$, 
define an increasing sequence $\{N_k\}_{k \in \Z_{\ge 0}}$ by setting
\begin{align}
N_{k } =  N_0^{\s^{k}}
\label{Z2}
\end{align}

\noi
for some $\s > 1$  such that 
\begin{align}
N_{k+1}^{2(1-s)} N_k^\al + N_k^{2\al} \ll N_{k+1}^\be, 
\label{Z3}
\end{align}

\noi
which requires $\be > 2(1-s)$.
Recalling that $\al > \be$ and \eqref{K12}, 
we have the following constraints:
\begin{align*}
2(1-s) < \be < \al \le 1 - 3(1-s),
\end{align*}

\noi
which imposes the condition $s > \frac 45$.
Suppose that 
\begin{align}
E(I_{N_k} \vec v)(t) \leq N^\al_k
\label{Z4}
\end{align}

\noi
for some $k$ and $t \geq 0 $.
Then,  by 
\eqref{I1}, Sobolev's inequality, \eqref{Z4}, and \eqref{Z3}, 
we have 
\begin{align}
\begin{split}
E(I_{N_{k+1}}\vec v)(t) 
& \les \|I_{N_{k+1}}\vec v\|_{\H^1}^2 +  \|I_{N_{k+1}}v\|_{L^4}^4 \\
& \les N_{k+1}^{2(1-s)} \|\vec v\|_{\H^{s}}^2 + \|v\|_{H^{\frac12}}^4 \\
& \les N_{k+1}^{2(1-s)} \|I_{N_k} \vec v\|_{\H^1}^2 + \|I_{N_k} v\|_{H^1}^4\\
&  \les  N_{k+1}^{2(1-s)}  E(I_{N_k}\vec v) + E(I_{N_k}\vec v)^2 \\
& \les  N_{k+1}^{2(1-s)} N_k^\al+ N_k^{2\al}\\
& \ll N_{k+1}^\be.
\end{split}
\label{Z5}
\end{align}

We are now ready to implement  an iterative argument.
Given $(\phi_0, \phi_1) \in \H^s(\T^2)$, choose $N_0 = N_0(\phi_0, \phi_1, s) \gg 1$
such that 
\begin{align} \label{eqn: sizeNk0}
E(I_{N_{0}}\vec v) (0) \leq N_{0}^\be.
\end{align}

\noi
By applying  Proposition \ref{PROP:LGWP}, we have
\begin{align*}
E(I_{N_{0}}\vec v) (t) \leq N_{0}^\al
\end{align*}

\noi
for any $0 \le t \le  \tau$.
By \eqref{Z4} and \eqref{Z5}, this then implies 
\begin{align*}
E(I_{N_{1}}\vec v) (\tau)\leq N_{1}^\be.
\end{align*}

\noi
Applying   Proposition \ref{PROP:LGWP} once again, 
we in turn  obtain
\begin{align*}
E(I_{N_{1}}\vec v) (t) \leq N_{1}^\al
\end{align*}

\noi
for $0 \le t \le  2 \tau$.
By \eqref{Z4} and \eqref{Z5}, this then implies 
\begin{align*}
E(I_{N_{2}}\vec v) (2\tau) \leq N_{2}^\be.
\end{align*}

\noi
After iterating this argument $\big[\frac{T}{\tau}\big] + 1$ times, we 
obtain a solution $v$ to 
 the renormalized cubic SNLW \eqref{SNLW6} on
 the time interval $[0, T]$.
Since the choice of $T \gg1 $ was arbitrary, this proves global well-posedness of \eqref{SNLW6}.

\begin{remark} \label{REM:bound}\rm

Fix $T\gg1 $ and let the other parameters be as above.
Then, it follows from the argument above
and  \eqref{I1} that 
\begin{align*}
\| \vec v(t)\|_{\H^s} \les
\big( E(I_{N_{k}}\vec v) (t)\big)^\frac{1}{2} \leq N_{k}^\frac{\al}{2}
\end{align*}

\noi
for any $ 0\le t \le T $ such that 
 $k \tau \leq t \leq (k+1) \tau$, 
$k \in \Z_{\ge 0}$.
Then, using \eqref{Z2}, we have
\begin{align}
\| \vec v(t)\|_{\H^s} \les
\exp\bigg( \frac{\al}{2} \s^k \log  N_{0}\bigg)
\le
\exp\bigg( \frac{\al}{2}  \log  N_{0}
\cdot \exp \Big(\frac{ (\log \s) t }{\tau}\Big)\bigg)
\label{Z9}
\end{align}

\noi
for $0 \leq t \leq T$.
Moreover, 
in view of \eqref{eqn: sizeNk0}, we
choose $N_0 \in \N$ such that 
$1 + E(I_{N_{0}}\vec v) (0) \sim N_{0}^\be$ and thus we have
\begin{align}
\log N_0 \sim \log 
\big(2 + \| \vec v(0)\|_{\H^s}\big).
\label{Z9a}
\end{align}

In order to reach the target time $T$, 
we  iteratively apply Proposition \ref{PROP:LGWP}
$K \sim \frac{T}{\tau}$-many times.
For this purpose, we need 
to guarantee the condition \eqref{K21}.
In view of 
Remark \ref{REM:tau}
and 
\eqref{Z2} with $k = K \sim \frac T\tau$, 
the condition \eqref{K21} now reads as
\begin{align*}
  e^{C(1 + V + R + T)   T^{-1} } - 1
\leq \frac{c_0 \s^{T^2} \log N_0 - \log 2}{T^2}, 
\end{align*}

\noi
which holds true
for any sufficiently large $T\gg1$.

Finally,  from 
\eqref{Z9}, \eqref{Z9a}, and Remark \ref{REM:tau}, 
we conclude the following double exponential bound:
\begin{align}
\| \vec v(t)\|_{\H^s} 
\le
C \exp\Big( c \log 
\big(2 + \| \vec v(0)\|_{\H^s}\big)
\cdot e^{C(\o)  t^2}\Big)
\label{Bd1}
\end{align}

\noi
for any $t \geq 0$.

\end{remark}

We conclude this section by pointing out that 
by implementing a more involved version of Proposition \ref{PROP:LGWP}
(see for example the paper \cite{Tolomeo2} by the fourth author,  studying SNLW on $\R^2$), 
it is possible to improve $t^2$ in 
\eqref{Bd1} to 
$t^{\al}$ for some $\al < 2$.
For readers' convenience, however, we decided to include
the current slightly simpler and more intuitive approach, 
with  a Gronwall-type argument  with $G \log G$ as in~\eqref{K15}. 
We point out that 
we
do not know how to improve $t^2$ in \eqref{Bd1} to $t$ at this point.

\section{Almost sure global well-posedness of the hyperbolic $\Phi_2$-model}
\label{SEC:GWP2}

We present a simple local well-posedness
argument for \eqref{SNLW13} based on Sobolev's inequality.
We first  consider the following deterministic NLW:
 \begin{align}
\begin{cases}
\dt^2 v + \dt v +(1-\Dl)v  +
 \sum_{\ell=0}^k {k\choose \ell} \, \Xi_\l  \, v^{k-\ell}
=0\\
(v,\dt v)|_{t = 0}=(v_0,v_1)
\end{cases}
\label{SNLW14}
\end{align}

\noi
for given initial data $(v_0, v_1)$ and a source $(\Xi_, \dots, \Xi_k)$
with the understanding that $\Xi_0 \equiv 1$.

Given $s\in \R$, define $\mathcal{X}^s(\T^2)$ by 
\[ \mathcal{X}^s(\T^2) \deff \H^s(\T^2) \times \big(L^2([0, 1]; W^{s-1, \infty}(\T^2))\big)^{\otimes k}\]

\noi
and set
\[ \|\pmb{\Xi}\|_{\mathcal{X}^s}
= \| (v_0, v_1) \|_{\H^s} + \sum_{j = 1}^k \| \Xi_j\|_{L^2([0, 1]; W^{s-1, \infty})}\]

\noi
for $\pmb{\Xi} = (v_0, v_1,  \Xi_1, \Xi_2, \dots, \Xi_k ) \in \mathcal{X}^s(\T^2)$.
Then, we have the following local well-posedness result for \eqref{SNLW14}.

\begin{proposition}\label{PROP:LWP}
Given an integer $k \geq 2$, 
there exists $\eps_k > 0$ such that,
for $0 \leq  \eps < \eps_k$,  
\eqref{SNLW14} is unconditionally local well-posed
in $\mathcal{X}^{1-\eps}(\T^2)$.
More precisely, 
given an enhanced data set: 
\begin{align}
\pmb{\Xi} = (v_0, v_1,  \Xi_1, \Xi_2, \dots, \Xi_k ) \in \mathcal{X}^{1-\eps}(\T^2), 
\label{data1}
\end{align}

\noi
there exist $T = T(\|\pmb{\Xi}\|_{\mathcal{X}^{1-\eps}}) \in (0, 1]$
and a unique solution $v$ to \eqref{SNLW14} in the class:
\begin{align}
C([0, T]; H^{1-\eps}(\T^2)).
\label{X1}
\end{align}
In particular, the uniqueness of $v$ 
holds in the entire class \eqref{X1}.
Furthermore, the solution map$:
\pmb{\Xi} \in \mathcal{X}^{1-\eps}(\T^2)
\mapsto 
v \in C([0, T; H^{1-\eps}(\T^2))$
is locally Lipschitz continuous.

\end{proposition}

We point out that 
Proposition~\ref{PROP:LWP}
is completely deterministic.
Once we prove Proposition~\ref{PROP:LWP}, 
the claimed local well-posedness of the renormalized SdNLW~\eqref{SNLW13}
follows
from 
Proposition~\ref{PROP:LWP} 
and 
Lemma \ref{LEM:Psi}, 
stating that 
the (random) enhanced data set
$\pmb{\Xi} = (v_0, v_1,  \Phi, :\!\Phi^2\!:, \dots, :\!\Phi^k\!:\, )$
almost surely belongs to 
 $\mathcal{X}^{1-\eps}(\T^2)$,
 $\eps > 0$.

\begin{proof}

By writing \eqref{SNLW14} in the Duhamel formulation, we
have 
 \begin{align}
 \begin{split}
v(t) = \G (v) \deff
\ &  \dt\D(t)v_0 + \D(t)(v_0+v_1)\\
  & +
\,  \sum_{\ell=0}^k  {k\choose \ell}
\int_0^t \D(t - t') 
\big( \Xi_\l\,  v^{k-\ell}\big)(t') dt', 
\end{split}
\label{SNLW15}
\end{align}

\noi
where the map $\G = \G_{\pmb{\Xi}}$ depends on the enhanced data set
$\pmb{\Xi}$ in \eqref{data1}.
Fix  $0 < T < 1$.

We first treat the case $\l =0$.
From  \eqref{lin2} and applying Sobolev's inequality twice, we obtain
\begin{align}
\begin{split}
\bigg\|\int_0^t \D(t - t') 
  v^k(t') dt'\bigg\|_{C_T H^{1-\eps}_x}
&   \les T \| v^k\|_{C_T H^{-\eps}_x}
\les T \| v^k\|_{C_T L^{\frac{2}{1+\eps}}_x}
 \les T \| v\|_{C_T L^{\frac{2k}{1+\eps}}_x}^k\\
& \les T \| v\|_{C_T H^s_x}^k, 
\end{split}
\label{X2}
\end{align}

\noi
provided that 
\begin{align*}
0 \leq \eps \leq \frac{1}{k-1}.
\end{align*}

\noi
For $1 \leq \l \leq k-1$, 
it follows from 
Lemma \ref{LEM:bilin}\,(ii)
and then (i) followed by Sobolev's inequality that 
\begin{align}
\begin{split}
\bigg\|\int_0^t \D(t - t') 
\big( \Xi_\l\,  v^{k-\ell}\big)(t') dt'\bigg\|_{C_T H^{1-\eps}_x}
& \les T^\frac{1}{2}\| \Xi_\l\,  v^{k-\ell}\|_{L^2_T H^{-\eps}_x}\\
& \les T^\frac{1}{2}\|\jb{\nb}^{-\eps} \Xi_\l
\|_{L^2_T L^\frac{2}{\eps}_x}
\| \jb{\nb}^\eps v^{k-\ell}\|_{C_T L^{2}_x}\\
& \les T^\frac{1}{2}\|\pmb{\Xi}\|_{\mathcal{X}^{1-\eps}}
\| \jb{\nb}^\eps v\|_{C_T L^{2(k-\l)}_x}^{k-\l}\\
& \les T^\frac{1}{2}\|\pmb{\Xi}\|_{\mathcal{X}^{1-\eps}}
\|  v\|_{C_T H^{1-\eps}_x}^{k-\l}, 
\end{split}
\label{X4}
\end{align}

\noi
provided that 
\begin{align}
0 \leq \eps \leq \frac 1{2(k-1)}.
\label{X5}
\end{align}

\noi
Lastly, from \eqref{lin2}, we have
\begin{align}
\begin{split}
\bigg\|\int_0^t \D(t - t') 
  \Xi_k(t') dt'\bigg\|_{C_T H^{1-\eps}_x}
&   \les T^\frac{1}{2} \| \Xi_k\|_{L^2_T H^{-\eps}_x}
\leq
 T^\frac{1}{2}\|\pmb{\Xi}\|_{\mathcal{X}^{1-\eps}}.
\end{split}
\label{X6}
\end{align}

Putting
 \eqref{SNLW15}, \eqref{X2}, \eqref{X4}, and \eqref{X6}
 together, we have
\begin{align*}
\|\G(v)\|_{C_T H^{1-\eps}_x}
&   \le
C_1 \| (v_0, v_1)\|_{\H^{1-\eps}}
+ 
C_2  T^\frac{1}{2}
 \big(1 + \|\pmb{\Xi}\|_{\mathcal{X}^{1-\eps}}\big)
\big( 1 +  \| v\|_{C_T H^{1-\eps}_x}\big)^k, 
\end{align*}

\noi
as long as \eqref{X5} is satisfied.
An analogous difference estimate also holds.
Therefore, 
by choosing
$T = T(\|\pmb{\Xi}\|_{\mathcal{X}^{1-\eps}}) >0$ sufficiently small, 
we conclude that $\G$ is a contraction in the ball 
$B_R \subset  C([0, T]; H^{1-\eps}(\T^2))$ of radius
$R = 2C_1 \| (v_0, v_1)\|_{\H^{1-\eps}} + 1$.
At this point, the uniqueness holds only in the ball $B_R$
but by a standard continuity argument, 
we can extend the uniqueness to hold
in the entire $C([0, T]; H^{1-\eps}(\T^2))$.
We omit details. 
\end{proof}

Next, we provide a brief discussion on invariance
of the truncated Gibbs measure 
$\rhoo_N$ in~\eqref{GibbsN}
under the dynamics of 
the   renormalized truncated SdNLW
\eqref{SNLW10}
for $u_N$.

Given $N \in \N$, define
the marginal probabilities measures
$\muu_{1, N}$ and $\muu_{1, N}^\perp$
on $\P_N \H^{-\eps}(\T^2)$
and $\P_N^\perp \H^{-\eps}(\T^2)$, respectively, 
as 
   the induced probability measures
under the following maps:
\begin{equation*}
\o \in \O \longmapsto (\P_N u^1(\o), \P_N u^2(\o))
 \end{equation*}

\noi
for 
$\muu_{1, N}$
and
\begin{equation*}
\o \in \O \longmapsto (\P_N^\perp u^1(\o), \P_N^\perp u^2(\o))
 \end{equation*}

\noi
for $\muu_{1, N}^\perp$, 
 where 
$u^1$ and $u^2$ are as in \eqref{series}.
Then, we have
\begin{align}
\muu_1 = 
\muu_{1, N} \otimes \muu_{1, N}^\perp.
\label{X7}
\end{align}

\noi
From \eqref{GibbsN} and \eqref{X7}, 
we then  have
\begin{align}
\rhoo_N = \vec \nu_{N} \otimes \muu_{1, N}^\perp.
\label{X8a}
\end{align}

\noi
where $\vec \nu_N$ is given by
\begin{align*}
d \vec \nu_N= Z_N^{-1}R_N(u)d\muu_{1, N}
\end{align*}

\noi
with the density $R_N$ as in \eqref{R1}.

Recalling  the decomposition \eqref{decomp2}, 
we see that
 the dynamics
for the high frequency part
$\P_N^\perp u_N = \P_N^\perp\Phi$
 is given by 
\begin{align}
\dt^2 \P_N^\perp\Phi + \dt  \P_N^\perp\Phi +(1-\Dl) \P_N^\perp\Phi  = \sqrt{2} \P_N^\perp\xi.
\label{X8}
\end{align}

\noi
This is a linear dynamics and thus we can readily verify
that the Gaussian measure $\muu_{1, N}^\perp$
is invariant under the dynamics of \eqref{X8}
(for example, by studying \eqref{X8} for each frequency $|n|> N$
on the Fourier side).

On the other hand, the low frequency part 
$\P_N u_N$ satisfies
\eqref{SNLW11a}.
With $(u_N^1, u_N^2) = (\P_N u_N, \dt \P_N u_N)$, we
can write
\eqref{SNLW11a}  in the following Ito formulation:
\begin{align}
\begin{split}
d  \begin{pmatrix}
u_N^1 \\ u_N^2
\end{pmatrix}
& + \Bigg\{
\begin{pmatrix}
0  & -1\\
1-\Dl &  0
\end{pmatrix}
 \begin{pmatrix}
u_N^1 \\ u_N^2
\end{pmatrix}
 +  
\begin{pmatrix}
0 \\ \P_N\big( :\!(u_N^1)^k\!:\big)
\end{pmatrix}
\Bigg\} dt \\
&   = 
  \begin{pmatrix}
0  \\ - u_N^2 dt + \sqrt 2\P_N dW
\end{pmatrix} .
\end{split}
\label{SNLW16}
\end{align}

\noi
This shows that the generator $\L^N$ for \eqref{SNLW16}
 can be written as $\L^N = \L^N_1 + \L^N_2$, 
 where $\L^N_1$ denotes the generator
 for the deterministic NLW with the truncated nonlinearity:
\begin{align}
\begin{split}
d  \begin{pmatrix}
u_N^1 \\ u_N^2
\end{pmatrix}
 + \Bigg\{
\begin{pmatrix}
0  & -1\\
1-\Dl &  0
\end{pmatrix}
 \begin{pmatrix}
u_N^1 \\ u_N^2
\end{pmatrix}
 +  
\begin{pmatrix}
0 \\ \P_N\big( :\!(u_N^1)^k\!:\big)
\end{pmatrix}
\Bigg\} dt 
   = 0 
\end{split}
\label{SNLW17}
\end{align}

\noi
and $\L^N_2$ denotes the generator
for the Ornstein-Uhlenbeck process 
(for the second component $u_N^2$):
\begin{align}
\begin{split}
d  \begin{pmatrix}
u_N^1 \\ u_N^2
\end{pmatrix}
   = 
  \begin{pmatrix}
0  \\ - u_N^2 dt + \sqrt 2\P_N dW
\end{pmatrix} .
\end{split}
\label{SNLW18}
\end{align}

Note that \eqref{SNLW17} is a Hamiltonian equation with the Hamiltonian:
\begin{align*}
E(u^1_N, u^2_N ) = \frac{1}{2}\int_{\T^2}\big( (u^1_N)^2 +  |\nb u^1_N|^2\big) dx
+ 
\frac{1}{2}\int_{\T^2} (u^2_N)^2dx
+ \log\big(R_N(u^1_N)\big), 
\end{align*}
 
\noi
where $R_N$ is as in \eqref{R1}.
Then, from  the conservation of the Hamiltonian 
$E(u^1_N, u^2_N )$ and Liouville's theorem
(on a finite-dimensional phase space), 
we conclude that $\vec \nu_N$ is invariant under the dynamics of \eqref{SNLW17}.
In particular, we have $(\L^N_1)^*\vec \nu_N = 0$.
On the other hand, 
by recalling that the Ornstein-Uhlenbeck process
preserves the standard Gaussian measure, 
we conclude that $\vec \nu_N$ is also invariant under the dynamics of \eqref{SNLW18}
since
 the measure $\vec \nu_N$ is nothing but the  white noise
(projected onto the low frequencies $\{|n|\leq N\}$)
on the second component $u_N^2$.
Thus, we have $(\L^N_2)^*\vec \nu_N = 0$.
Hence, we obtain
\[(\L^N)^*\vec \nu_N = 
(\L^N_1)^*\vec \nu_N +  (\L^N_2)^*\vec  \nu_N = 0.\]

\noi
This shows invariance of $\vec \nu_N$ under \eqref{SNLW16}
and hence under \eqref{SNLW11a}.

Therefore, from \eqref{X8a}
and invariance of $\vec \nu_N$ and $\muu^\perp_{1, N}$ 
under \eqref{SNLW16} and \eqref{X8}, respectively,  we conclude that  the truncated Gibbs measure 
$\rhoo_N$ in~\eqref{GibbsN}
is invariant 
under the dynamics of 
the   renormalized truncated SdNLW
\eqref{SNLW10}.

The rest of the proof of 
Theorem \ref{THM:GWP2} 
follows from a standard application of Bourgain's invariant measure argument
and thus we omit details.
See, for example,~\cite{ORTz}
for details.

\begin{ackno}\rm
T.O.~was supported by the European Research Council (grant no.~637995 ``ProbDynDispEq'' and grant no.~864138 ``SingStochDispDyn''). L.T.~was supported by the European Research Council (grant no. 637995 ``ProbDynDispEq''). M.G., H.K., and L.T ~were supported by the Deutsche Forschungsgemeinschaft (DFG, German Research Foundation) through the Hausdorff Center for Mathematics under Germany's Excellence Strategy - EXC-2047/1 - 390685813 and through CRC 1060 - project number 211504053.
%
The authors would like to thank the anonymous referees for helpful comments.
\end{ackno}


\begin{thebibliography}{99}



\bibitem{russo4}
S.~Albeverio, Z.~Haba, F.~Russo, 
{\it Trivial solutions for a non-linear two-space-dimensional wave equation perturbed by space-time white noise,} Stochastics Stochastics Rep.  56 (1996), no.~1-2, 127--160. 




\bibitem{BOP2}
\'A.~B\'enyi, T.~Oh, O.~Pocovnicu, 
{\it On the probabilistic Cauchy theory of the cubic nonlinear Schr\"odinger equation on $\R^d$, $d \geq 3$},  Trans. Amer. Math. Soc. Ser. B 2 (2015), 1--50. 



\bibitem{BO94}
J.~Bourgain, 
{\it Periodic nonlinear Schr\"odinger equation and invariant measures}, 
Comm. Math. Phys. 166 (1994), no. 1, 1--26.

\bibitem{BO96}
J.~Bourgain, 
{\it Invariant measures for the 2D-defocusing nonlinear Schr\"odinger equation}, 
Comm. Math. Phys. 176 (1996), no. 2, 421--445. 



\bibitem{Bring}
B.~Bringmann, 
{\it Invariant Gibbs measures for the three-dimensional wave equation
with a Hartree nonlinearity II: dynamics}, 
arXiv:2009.04616 [math.AP].




\bibitem{BT2}
N.~Burq, N.~Tzvetkov, 
{\it Probabilistic well-posedness for the cubic wave equation,} J. Eur. Math. Soc. 16 (2014), no. 1, 1--30. 



\bibitem{CLO}
K.~Cheung, G.~Li, T.~Oh,
{\it  Almost conservation laws for stochastic nonlinear Schr\"odinger equations},
J. Evol. Equ.
 https://doi.org/10.1007/s00028-020-00659-x 

 
 
 
\bibitem{CKSTT1}
J.~Colliander, M.~Keel, G.~Staffilani, H.~Takaoka, T.~Tao, 
{\it  Almost conservation laws and global rough solutions to a nonlinear Schr\"odinger equation,} Math. Res. Lett. 9 (2002), no. 5-6, 659--682.







\bibitem{CKSTT2}
J.~Colliander, M.~Keel, G.~Staffilani, H.~Takaoka, T.~Tao, 
{\it Sharp global well-posedness for KdV and modified KdV on $ \R $ and $ \T $},
 J. Amer. Math. Soc. 16 (2003), no. 3, 705--749.


\bibitem{CO}
J.~Colliander, T.~Oh, {\it  Almost sure well-posedness of the cubic nonlinear Schr\"odinger equation below $L^2(\T)$}, Duke Math. J. 161 (2012), no. 3, 367--414. 



\bibitem{DPD}
G.~Da Prato, A.~Debussche, 
{\it Strong solutions to the stochastic quantization equations,} Ann. Probab. 31 (2003), no. 4, 1900--1916.


%
\bibitem{DPZ14}
G.~Da~Prato, J.~Zabczyk,
\textit{Stochastic equations in infinite dimensions},
Second edition. Encyclopedia of Mathematics and its Applications, 152. Cambridge University Press, Cambridge, 2014. xviii+493 pp.


\bibitem{DPT1}
G.~Da Prato, L.~Tubaro, 
{\it Wick powers in stochastic PDEs: an introduction,} 
Technical Report UTM, 2006, 39 pp.



\bibitem{DD2}
A.~de~Bouard, A.~Debussche,
\textit{The stochastic nonlinear Schr\"odinger equation in $H^1$},
Stochastic Anal. Appl. 21 (2003), no. 1, 97--126.



\bibitem{Deya1}
A.~Deya, {\it A nonlinear wave equation with fractional perturbation,} 
Ann. Probab. 47 (2019), no. 3, 1775--1810.


\bibitem{Deya2}
A.~Deya,
{\it On a non-linear 2D fractional wave equation},
 Ann. Inst. Henri Poincar\'e Probab. Stat. 56 (2020), no. 1, 477--501.




\bibitem{Forlano}
J.~Forlano, 
{\it Almost sure global well posedness for the BBM equation with infinite $L^2$ initial data},
 Discrete Contin. Dyn. Syst. 40 (2020), no. 1, 267--318.



\bibitem{FV}
P.~Friz, N.~Victoir, 
{\it Multidimensional stochastic processes as rough paths. Theory and applications,}
Cambridge Studies in Advanced Mathematics, 120. Cambridge University Press, Cambridge, 2010. xiv+656 pp.



\bibitem{GJ}
J.~Glimm, A.~Jaffe, 
{\it Quantum physics. A functional integral point of view,} Second edition. Springer-Verlag, New York, 1987. xxii+535 pp.


\bibitem{GH}
M.~Gubinelli, M.~Hofmanov\'a,
{\it Global solutions to elliptic and parabolic $\Phi^4$ models in Euclidean space,}
 Comm. Math. Phys. 368 (2019), no. 3, 1201--1266. 

\bibitem{GKO1}
M.~Gubinelli, H.~Koch, T.~Oh,
{\it Renormalization of the two-dimensional stochastic nonlinear wave equations,}
 Trans. Amer. Math. Soc. 370 (2018), no 10, 7335--7359. 



\bibitem{GKO2}
M.~Gubinelli, H.~Koch, T.~Oh,
{\it Paracontrolled approach to the three-dimensional stochastic nonlinear wave equation with quadratic nonlinearity}, 
to appear in J. Eur. Math. Soc.









\bibitem{Kato}
T.~Kato, {\it On nonlinear Schr\"odinger equations. II. $H^s$-solutions and unconditional well-posedness,}
 J. Anal. Math. 67 (1995), 281--306 

\bibitem{KeelTao}
M.~Keel, T.~Tao, {\it Endpoint Strichartz estimates},  Amer. J. Math. 120 (1998), no. 5, 955--980.


\bibitem{Kuo}
H.~Kuo, 
{\it Introduction to stochastic integration,} Universitext. Springer, New York, 2006. xiv+278 pp.




\bibitem{McKean}
H.P.~McKean, 
{\it Statistical mechanics of nonlinear wave equations. IV. Cubic Schr\"odinger,} 
 Comm. Math. Phys. 168 (1995), no. 3, 479--491. 
 {\it Erratum: Statistical mechanics of nonlinear wave equations. IV. Cubic Schr\"odinger}, Comm. Math. Phys. 173 (1995), no. 3, 675.


\bibitem{MoinatW}
A.~Moinat, H.~Weber
{\it Space-time localisation for the dynamic $\Phi^4_3$ model,}
 Comm. Pure Appl. Math. 73 (2020), no. 12, 2519--2555.


\bibitem{MPTW}
R.~Mosincat, O.~Pocovnicu, L.~Tolomeo, Y.~Wang,
{\it  Global well-posedness of three-dimensional periodic stochastic nonlinear beam equations,}
 preprint. 
 
 
\bibitem{MW1}
J.-C.~Mourrat, H.~Weber,
{\it Global well-posedness of the dynamic $\Phi^4$ model in the plane,}
 Ann. Probab. 45 (2017), no. 4, 2398--2476.



\bibitem{MW2}
J.-C.~Mourrat, H.~Weber,
{\it The dynamic $\Phi^4_3$ model comes down from infinity}, 
Comm. Math. Phys. 356 (2017), no. 3, 673--753.





\bibitem{Nelson2}
E.~Nelson, 
{\it A quartic interaction in two dimensions}, 
 1966 Mathematical Theory of Elementary Particles (Proc. Conf., Dedham, Mass., 1965) pp. 69--73 M.I.T. Press, Cambridge, Mass.



\bibitem{Nu}
D.~Nualart, 
{\it The Malliavin calculus and related topics,} Second edition. Probability and its Applications (New York). Springer-Verlag, Berlin, 2006. xiv+382 pp.





\bibitem{OO}
T.~Oh, M.~Okamoto, 
{\it Comparing the stochastic nonlinear wave and heat equations: a case study}, 
 Electron. J. Probab.
 26 (2021), paper no. 9, 44 pp.




\bibitem{OOR}
T.~Oh, M.~Okamoto, T.~Robert,
{\it  A remark on triviality for the two-dimensional stochastic nonlinear wave equation,}
 Stochastic Process. Appl.
 130 (2020), no. 9, 5838--5864. 

\bibitem{OOT1}
T.~Oh, M.~Okamoto, L.~Tolomeo, 
{\it 
Focusing $\Phi^4_3$-model with a Hartree-type nonlinearity}, 
arXiv:2009.03251 [math.PR].


\bibitem{OOT2}
T.~Oh, M.~Okamoto, L.~Tolomeo, 
{\it Stochastic quantization of the $\Phi^3_3$-model}, 
preprint.


\bibitem{OOTz}
T.~Oh, M.~Okamoto, N.~Tzvetkov,
{\it  Uniqueness and non-uniqueness of the Gaussian free field evolution under the two-dimensional Wick ordered cubic wave equation,}
 preprint. 

\bibitem{OP1}
T.~Oh, O.~Pocovnicu,
{\it  Probabilistic global well-posedness of the energy-critical defocusing quintic nonlinear wave equation 
on $\R^3$}, J. Math. Pures Appl. 105 (2016), 342--366. 





\bibitem{OPTz}
T.~Oh, O.~Pocovnicu, N.~Tzvetkov,
{\it  Probabilistic local Cauchy theory of the cubic nonlinear wave equation in negative Sobolev spaces},
to appear in Ann. Inst. Fourier (Grenoble). 

\bibitem{ORTz}
T.~Oh, T.~Robert, N.~Tzvetkov,
{\it  Stochastic nonlinear wave dynamics on compact surfaces},
arXiv:1904.05277 [math.AP].


\bibitem{ORSW1}
T.~Oh, T.~Robert, P.~Sosoe, Y.~Wang,
{\it  On the two-dimensional hyperbolic stochastic sine-Gordon equation}, 
 Stoch. Partial Differ. Equ. Anal. Comput.
 9 (2021), 1--32. 
%

\bibitem{ORSW2}
T.~Oh, T.~Robert, P.~Sosoe, Y.~Wang,
{\it Invariant Gibbs dynamics for the dynamical sine-Gordon model},
 Proc. Roy. Soc. Edinburgh Sect. A (2020), 17 pages. doi: https://doi.org/10.1017/prm.2020.68
 
 
 \bibitem{ORW}
T.~Oh, T.~Robert, Y.~Wang,
{\it  On the parabolic and hyperbolic Liouville equations},
to appear in Comm. Math. Phys.
 
   
\bibitem{OT1}
T.~Oh, L.~Thomann, 
{\it 
 A pedestrian approach to the invariant Gibbs measure for 
the 2-$d$ defocusing nonlinear Schr\"odinger equations}, 
 6 (2018), 397--445. 


\bibitem{OT2}
T.~Oh, L.~Thomann, 
{\it Invariant Gibbs measure for 
the 2-$d$ defocusing nonlinear wave  equations}, 
Ann. Fac. Sci. Toulouse Math.
 29 (2020), no. 1, 1--26. 

\bibitem{Poc}
O.~Pocovnicu, 
{\it  Almost sure global well-posedness for the energy-critical defocusing nonlinear wave equation 
on $\R^d$, $d = 4$ and $5$}, J. Eur. Math. Soc. 19 (2017), 2321--2375. 

\bibitem{Roy}
T.~Roy, 
{\it On the interpolation with the potential bound for global solutions of the defocusing cubic wave equation on $\T^2$},  J. Funct. Anal. 270 (2016), no. 9, 3280--3306.



\bibitem{RSS}
S.~Ryang, T.~Saito, K.~Shigemoto,
{\it Canonical stochastic quantization},
Progr. Theoret. Phys. 73 (1985), no. 5, 1295--1298. 

\bibitem{Shige}
I.~Shigekawa, 
{\it Stochastic analysis,}
Translated from the 1998 Japanese original by the author. Translations of Mathematical Monographs, 224. Iwanami Series in Modern Mathematics. American Mathematical Society, Providence, RI, 2004. xii+182 pp.


\bibitem{Simon}
B.~Simon, 
{\it  The $P(\varphi)_2$ Euclidean (quantum) field theory,} Princeton Series in Physics. Princeton University Press, Princeton, N.J., 1974. xx+392 pp.



\bibitem{TTz}
L.~Thomann, N.~Tzvetkov, 
{\it Gibbs measure for the periodic derivative nonlinear Schr\"odinger equation},
Nonlinearity 23 (2010), no. 11, 2771--2791.


\bibitem{Tolomeo1}
L.~Tolomeo, 
{\it 
Unique ergodicity for a class of stochastic hyperbolic equations with additive space-time white noise}, 
 Comm. Math. Phys. 377 (2020), no. 2, 1311--1347.

\bibitem{Tolomeo2}
L.~Tolomeo, 
{\it Global well-posedness of the two-dimensional stochastic nonlinear wave equation on an unbounded domain}, 
 Ann. Probab. 49 (2021), no. 3, 
 1402--1426.

\bibitem{Tolomeo3}
L.~Tolomeo, 
{\it Ergodicity for the hyperbolic $P(\Phi)_2$-model}, 
in preparation.



\bibitem{TzBO}
N.~Tzvetkov,  {\it Construction of a Gibbs measure associated to the periodic Benjamin-Ono equation,}
Probab. Theory Related Fields 146 (2010), no. 3-4, 481--514.




\end{thebibliography}
\end{document}